\documentclass[12pt, a4paper, reqno]{amsart}
\usepackage{amsmath}
\usepackage{amsfonts}
\usepackage{amssymb}
\usepackage{color}
\usepackage{comment,graphicx}
\usepackage[T1]{fontenc}
\usepackage{url}
\usepackage{ae}

\def\phi{\varphi}
\def\rho{\varrho}
\def\epsilon{\varepsilon}
\numberwithin{equation}{section}
\theoremstyle{plain}
\newtheorem{theorem}[equation]{Theorem}
\newtheorem{lemma}[equation]{Lemma}

\theoremstyle{definition}
\newtheorem{definition}[equation]{Definition}

\theoremstyle{remark}

\renewcommand{\ln}{\log}

\renewcommand{\leq}{\leqslant}
\renewcommand{\geq}{\geqslant}

\input{tcilatex}

\begin{document}
\title[Continuous characterization of \ $B_{p(\cdot ),q(\cdot )}^{\alpha
(\cdot )}$ spaces]{Continuous characterization of Besov spaces of variable
smoothness and integrability}
\author[ S. Ben Mahmoud and D. Drihem]{Salah Ben Mahmoud and Douadi Drihem }

\begin{abstract}
In this paper we obtain new equivalent quasi-norms of the Besov spaces of
variable smoothness and integrability. Our main tools are the continuous
version of Calder\'{o}n reproducing formula, maximal inequalities and
variable exponent technique, but allowing the parameters to vary from point
to point will raise extra difficulties which, in general, are overcome by
imposing regularity assumptions on these exponents.\vskip5pt
\end{abstract}

\date{\today }
\keywords{Besov space, variable exponent, Calder\'{o}n reproducing formula.\\
\\
Salah Ben Mahmoud and Douadi Drihem\\
M'sila University, Department of Mathematics, Laboratory of Functional
Analysis and Geometry of Spaces , P.O. Box 166, M'sila 28000, Algeria,
e-mail:\texttt{\ salahmath2016@gmail.com (Salah Ben Mahmoud,} \texttt{%
douadidr@yahoo.fr, douadi.drihem@univ-msila.dz (Douadi Drihem),.}}
\maketitle

\section{Introduction}

Besov spaces of variable smoothness and integrability initially appeared in
the paper of A. Almeida and P. H\"{a}st\"{o} \cite{AH},\ where several basic
properties were shown, such as the Fourier analytical characterization.
Later the author \cite{D3} characterized these spaces by local means and
established the atomic characterization. After that, Kempka and Vyb\'{\i}ral 
\cite{KV122} characterized these spaces by ball means of differences and
also by local means. The duality of these function spaces is given in \cite%
{IN14} and \cite{N14}.

The interest in these spaces comes not only from theoretical reasons but
also from their applications to several classical problems in analysis. For
further considerations of PDEs, we refer to \cite{DHHR} and references
therein.

The main aim of this paper is to present new equivalent quasi-norm of these
function spaces, which based on the continuous version of Calder\'{o}n
reproducing formula. Firstly, we define new family of function spaces and
prove their basic properties. Secondly, under some suitable assumptions on
the parameters we prove that these function spaces are just the Besov spaces
of variable smoothness and integrability of Almeida and H\"{a}st\"{o}.
Finally we characterize these function spaces in terms of continuous local
means.

This paper needs some notation. As usual, we denote by $\mathbb{N}_{0}$ the
set of all non-negative integers. The notation $f\lesssim g$ means that $%
f\leq c\,g$ for some independent positive constant $c$ (and non-negative
functions $f$ and $g$), and $f\approx g$ means that $f\lesssim g\lesssim f$.
For $x\in \mathbb{R}$, $\lfloor x\rfloor $ stands for the largest integer
smaller than or equal to $x$.\vskip5pt

If $E\subset {\mathbb{R}^{n}}$ is a measurable set, then $|E|$ stands for
the Lebesgue measure of $E$ and $\chi _{E}$ denotes its characteristic
function. By $c$ we denote generic positive constants, which may have
different values at different occurrences. Although the exact values of the
constants are usually irrelevant for our purposes, sometimes we emphasize
their dependence on certain parameters (e.g., $c(p)$ means that $c$ depends
on $p$, etc.).

The symbol $\mathcal{S}(\mathbb{R}^{n})$ is used in place of the set of all
Schwartz functions on $\mathbb{R}^{n}$. We define the Fourier transform of a
function $f\in \mathcal{S}(\mathbb{R}^{n})$ by 
\begin{equation*}
\mathcal{F}(f)(\xi ):=(2\pi )^{-n/2}\int_{\mathbb{R}^{n}}e^{-ix\cdot \xi
}f(x)dx,\quad \xi \in \mathbb{R}^{n}.
\end{equation*}

We denote by $\mathcal{S}^{\prime }(\mathbb{R}^{n})$ the dual space of all
tempered distributions on $\mathbb{R}^{n}$. The variable exponents that we
consider are always measurable functions $p$ on $\mathbb{R}^{n}$ with range
in $(0,\infty ]$. We denote by $\mathcal{P}_{0}(\mathbb{R}^{n})$ the set of
such functions bounded away from the origin (i.e., $p^{-}>0$). The subset of
variable exponents with range in $[1,\infty ]$ is denoted by $\mathcal{P}(%
\mathbb{R}^{n})$. We use the standard notation: 
\begin{equation*}
p^{-}:=\underset{x\in \mathbb{R}^{n}}{\text{ess-inf}}\,p(x)\quad \text{and}%
\quad p^{+}:=\underset{x\in \mathbb{R}^{n}}{\text{ess-sup}}\,p(x).
\end{equation*}%
We put%
\begin{equation*}
\omega _{p}(t)=\left\{ 
\begin{array}{ccc}
t^{p} & \text{if} & p\in (0,\infty )\text{ and }t>0, \\ 
0 & \text{if} & p=\infty \text{ and }0<t\leq 1, \\ 
\infty & \text{if} & p=\infty \text{ and }t>1.%
\end{array}%
\right.
\end{equation*}%
The variable exponent modular is defined by 
\begin{equation*}
\varrho _{p(\cdot )}(f):=\int_{\mathbb{R}^{n}}\omega _{p(x)}(|f(x)|)\,dx.
\end{equation*}%
The variable exponent Lebesgue space $L^{p(\cdot )}$\ consists of measurable
functions $f$ on $\mathbb{R}^{n}$ such that $\varrho _{p(\cdot )}(\lambda
f)<\infty $ for some $\lambda >0$. We define the Luxemburg (quasi)-norm on
this space by the formula 
\begin{equation*}
\big\|f\big\|_{p(\cdot )}:=\inf \Big\{\lambda >0:\varrho _{p(\cdot )}\Big(%
\frac{f}{\lambda }\Big)\leq 1\Big\}.
\end{equation*}%
A useful property is that $\big\|f\big\|_{p(\cdot )}\leq 1$ if and only if $%
\varrho _{p(\cdot )}(f)\leq 1$ (see Lemma 3.2.4 from \cite{DHHR}). \vskip5pt

Let $p,q\in \mathcal{P}_{0}(\mathbb{R}^{n})$. The mixed Lebesgue-sequence
space $\ell ^{q(\cdot )}(L^{p(\cdot )})$ is defined on sequences of $%
L^{p(\cdot )}$-functions by the modular 
\begin{equation*}
\varrho _{\ell ^{q(\cdot )}(L^{p\left( \cdot \right)
})}((f_{v})_{v}):=\sum_{v=0}^{\infty }\inf \Big\{\lambda _{v}>0:\varrho
_{p(\cdot )}\Big(\frac{f_{v}}{\lambda _{v}^{1/q(\cdot )}}\Big)\leq 1\Big\}.
\end{equation*}%
The (quasi)-norm is defined from this as usual:%
\begin{equation}
\big\|(f_{v})_{v}\big\|_{\ell ^{q(\cdot )}(L^{p\left( \cdot \right)
})}:=\inf \Big\{\mu >0:\varrho _{\ell ^{q(\cdot )}(L^{p(\cdot )})}\Big(\frac{%
1}{\mu }(f_{v})_{v}\Big)\leq 1\Big\}.  \label{mixed-norm}
\end{equation}%
If $q^{+}<\infty $, then we can replace \eqref{mixed-norm} by a simpler
expression: 
\begin{equation*}
\varrho _{\ell ^{q(\cdot )}(L^{p(\cdot )})}((f_{v})_{v})=\sum_{v=0}^{\infty }%
\big\||f_{v}|^{q(\cdot )}\big\|_{\frac{p(\cdot )}{q(\cdot )}}.
\end{equation*}%
We use this notation even when $q^{+}=\infty $. Let $(f_{t})_{0<t\leq 1}$\
be a sequence of measurable functions when $t$ is a continuous variable. We
set%
\begin{equation*}
\varrho _{\widetilde{\ell ^{q(\cdot )}(L^{p\left( \cdot \right) })}%
}((f_{t})_{0<t\leq 1}):=\int_{0}^{1}\inf \big\{\lambda _{t}:\varrho
_{p(\cdot )}\big(\frac{f_{t}}{\lambda _{t}^{1/q(\cdot )}}\big)\leq 1\big\}%
\frac{dt}{t}.
\end{equation*}%
The (quasi)-norm is defined by%
\begin{equation*}
\big\|(f_{t})_{0<t\leq 1}\big\|_{\widetilde{\ell ^{q(\cdot )}(L^{p\left(
\cdot \right) })}}:=\inf \Big\{\mu >0:\varrho _{\widetilde{\ell ^{q(\cdot
)}(L^{p\left( \cdot \right) })}}\Big(\frac{1}{\mu }(f_{t})_{0<t\leq 1}\Big)%
\leq 1\Big\}.
\end{equation*}

We say that a real valued-function $g$ on $\mathbb{R}^{n}$ is \textit{%
locally }log\textit{-H\"{o}lder continuous} on $\mathbb{R}^{n}$, abbreviated 
$g\in C_{\text{loc}}^{\log }(\mathbb{R}^{n})$, if there exists a constant $%
c_{\log }(g)>0$ such that 
\begin{equation}
\left\vert g(x)-g(y)\right\vert \leq \frac{c_{\log }(g)}{\log
(e+1/\left\vert x-y\right\vert )}  \label{lo-log-Holder}
\end{equation}%
for all $x,y\in \mathbb{R}^{n}$.

We say that $g$ satisfies the log\textit{-H\"{o}lder decay condition}, if
there exist two constants $g_{\infty }\in \mathbb{R}$ and $c_{\log }>0$ such
that%
\begin{equation*}
\left\vert g(x)-g_{\infty }\right\vert \leq \frac{c_{\log }}{\log
(e+\left\vert x\right\vert )}
\end{equation*}%
for all $x\in \mathbb{R}^{n}$. We say that $g$ is \textit{globally} log%
\textit{-H\"{o}lder continuous} on $\mathbb{R}^{n}$, abbreviated $g\in
C^{\log }(\mathbb{R}^{n})$, if it is\textit{\ }locally log-H\"{o}lder
continuous on $\mathbb{R}^{n}$ and satisfies the log-H\"{o}lder decay\textit{%
\ }condition.\textit{\ }The constants $c_{\log }(g)$ and $c_{\log }$ are
called the \textit{locally }log\textit{-H\"{o}lder constant } and the log%
\textit{-H\"{o}lder decay constant}, respectively\textit{.} We note that any
function $g\in C_{\text{loc}}^{\log }(\mathbb{R}^{n})$ always belongs to $%
L^{\infty }$.\vskip5pt

We define the following class of variable exponents: 
\begin{equation*}
\mathcal{P}_{0}^{\mathrm{log}}(\mathbb{R}^{n}):=\Big\{p\in \mathcal{P}_{0}(%
\mathbb{R}^{n}):\frac{1}{p}\in C^{\log }(\mathbb{R}^{n})\Big\},
\end{equation*}%
which is introduced in \cite[Section 2]{DHHMS}. The class $\mathcal{P}^{%
\mathrm{log}}(\mathbb{R}^{n})$ is defined analogously. We define 
\begin{equation*}
\frac{1}{p_{\infty }}:=\lim_{|x|\rightarrow \infty }\frac{1}{p(x)},
\end{equation*}%
and we use the convention $\frac{1}{\infty }=0$. Note that although $\frac{1%
}{p}$ is bounded, the variable exponent $p$ itself can be unbounded. We put 
\begin{equation*}
\Psi \left( x\right) :=\sup_{\left\vert y\right\vert \geq \left\vert
x\right\vert }\left\vert \varphi \left( y\right) \right\vert 
\end{equation*}%
for $\varphi \in L^{1}$. We suppose that $\Psi \in L^{1}$. Then it was
proved in \cite[Lemma \ 4.6.3]{DHHR} that if $p\in \mathcal{P}^{\mathrm{log}%
}(\mathbb{R}^{n})$, then 
\begin{equation*}
\big\|\varphi _{\varepsilon }\ast f\big\|_{{p(\cdot )}}\leq c\big\|\Psi %
\big\|_{{1}}\big\|f\big\|_{{p(\cdot )}}
\end{equation*}%
for all $f\in L^{p(\cdot )}$, where 
\begin{equation*}
\varphi _{\varepsilon }:=\frac{1}{\varepsilon ^{n}}\varphi \left( \frac{%
\cdot }{\varepsilon }\right) ,\quad \varepsilon >0.
\end{equation*}%
We put 
\begin{equation*}
\eta _{t,m}(x):=t^{-n}(1+t^{-1}\left\vert x\right\vert )^{-m}
\end{equation*}%
for any $x\in \mathbb{R}^{n}$, $t>0$ and $m>0$. Note that $\eta _{t,m}\in
L^{1}$ when $m>n$ and that $\big\|\eta _{t,m}\big\|_{1}=c(m)$ is independent
of $t$. If $t=2^{-v}$, $v\in \mathbb{N}_{0}$ then we put 
\begin{equation*}
\eta _{v,m}:=\eta _{2^{-v},m}.
\end{equation*}%
We refer to the recent monograph \cite{CF13} for further properties,
historical remarks and references on variable exponent spaces.

\section{Basic tools}

In this section we present some useful results. The following lemma is
proved in \cite[Lemma 6.1]{DHR} (see also \cite[Lemma 19]{KV122}).

\begin{lemma}
\label{DHR-lemma}Let $\alpha \in C_{\mathrm{loc}}^{\log }(\mathbb{R}^{n})$, $%
m\in \mathbb{N}_{0}$ and let $R\geq c_{\log }(\alpha )$, where $c_{\log
}(\alpha )$ is the constant from \eqref{lo-log-Holder} for $g=\alpha $. Then
there exists a constant $c>0$ such that 
\begin{equation*}
t^{-\alpha (x)}\eta _{t,m+R}(x-y)\leq c\text{ }t^{-\alpha (y)}\eta
_{t,m}(x-y)
\end{equation*}%
for any $0<t\leq 1$ and $x,y\in \mathbb{R}^{n}$.
\end{lemma}

The previous lemma allows us to treat the variable smoothness in many cases
as if it were not variable at all. Namely, we can move the factor $%
t^{-\alpha (x)}$ inside the convolution as follows: 
\begin{equation*}
t^{-\alpha (x)}\eta _{t,m+R}\ast f(x)\leq c\text{ }\eta _{t,m}\ast
(t^{-\alpha (\cdot )}f)(x).
\end{equation*}

The following lemma is from \cite[Lemma 3.14]{YZW15}.

\begin{lemma}
\label{DZW}Let $p,q\in \mathcal{P}_{0}(\mathbb{R}^{n})$. Let $f$ be a
measurable function on $\mathbb{R}^{n}$. If 
\begin{equation*}
\big\||f|^{q(\cdot )}\big\|_{\frac{p(\cdot )}{q(\cdot )}}\geq 1
\end{equation*}%
then 
\begin{equation*}
\big\|f\big\|_{p(\cdot )}^{q^{-}}\leq \big\||f|^{q(\cdot )}\big\|_{\frac{%
p(\cdot )}{q(\cdot )}}.
\end{equation*}
\end{lemma}

The next lemma is a Hardy type inequality, see \cite{Mo87}.

\begin{lemma}
\label{Hardy-inequality1}Let $s>0$ and $(\varepsilon _{t})_{0<t\leq 1}$ be a
sequence of positive measurable functions when $t$ is a continuous variable.
Let%
\begin{equation*}
\eta _{t}=t^{s}\int_{t}^{1}\tau ^{-s}\varepsilon _{\tau }\frac{d\tau }{\tau }%
\quad \text{and\quad }\delta _{t}=t^{-s}\int_{0}^{t}\tau ^{s}\varepsilon
_{\tau }\frac{d\tau }{\tau }.
\end{equation*}%
Then there exists a constant $c>0\ $\textit{depending only on }$s$ such that%
\begin{equation*}
\int_{0}^{1}\eta _{t}\frac{dt}{t}+\int_{0}^{1}\delta _{t}\frac{dt}{t}\leq
c\int_{0}^{1}\varepsilon _{t}\frac{dt}{t}.
\end{equation*}
\end{lemma}

\begin{lemma}
\label{r-trick}Let $r,N>0$, $m>n$ and $\theta ,\omega \in \mathcal{S}\left( 
\mathbb{R}^{n}\right) $ with $\mathrm{supp}\,\mathcal{F}\omega \subset 
\overline{B(0,1)}$. Then there exists a constant $c=c(r,m,n)>0$ such that
for all $g\in \mathcal{S}^{\prime }\left( \mathbb{R}^{n}\right) $, we have%
\begin{equation*}
\left\vert \theta _{N}\ast \omega _{N}\ast g\left( x\right) \right\vert \leq
c(\eta _{N,m}\ast \left\vert \omega _{N}\ast g\right\vert
^{r}(x))^{1/r},\quad x\in \mathbb{R}^{n},
\end{equation*}%
where $\theta _{N}(\cdot ):=N^{n}\theta (N\cdot )$, $\omega _{N}(\cdot
):=N^{n}\omega (N\cdot )$ and $\eta _{N,m}:=N^{n}(1+N\left\vert \cdot
\right\vert )^{-m}$.
\end{lemma}

The proof of this lemma is given in \cite[Lemma 2.2]{D6}. The following
lemma is from A. Almeida and P. H\"{a}st\"{o} \cite[Lemma 4.7]{AH} (we use
it, since the maximal operator is in general not bounded on $\ell ^{q(\cdot
)}(L^{p(\cdot )})$, see \cite[Example 4.1]{AH}).

\begin{lemma}
\label{Alm-Hastolemma1}Let $p\in \mathcal{P}^{\log }(\mathbb{R}^{n})$ and $%
q\in \mathcal{P}_{0}\left( \mathbb{R}^{n}\right) $\ with\ $\frac{1}{q}\in C_{%
\mathrm{loc}}^{\log }\left( \mathbb{R}^{n}\right) $. For $m>n+c_{\log }(1/q)$%
, there exists $c>0$ such that%
\begin{equation*}
\left\Vert \left( \eta _{v,m}\ast f_{v}\right) _{v}\right\Vert _{\ell
^{q(\cdot )}(L^{p(\cdot )})}\leq c\left\Vert \left( f_{v}\right)
_{v}\right\Vert _{\ell ^{q(\cdot )}(L^{p(\cdot )})}.
\end{equation*}
\end{lemma}

\begin{lemma}
\label{Al-Ha-Lemma 1}Let $0<\alpha <\beta <\infty ,p\in \mathcal{P}^{\log
}\left( \mathbb{R}^{n}\right) \ $and $q\in \mathcal{P}\left( \mathbb{R}%
^{n}\right) $\ with\ $\frac{1}{q}\in C_{\mathrm{loc}}^{\log }\left( \mathbb{R%
}^{n}\right) $. Let 
\begin{equation*}
g_{t}(x):=\int_{\alpha t}^{\beta t}\eta _{\tau ,m}\ast f_{\tau }(x)\frac{%
d\tau }{\tau },\quad t\in (0,1],x\in \mathbb{R}^{n}.
\end{equation*}%
$\mathrm{(i)}$ Assume that $0<\beta t\leq 1$. The inequality%
\begin{equation*}
\big\||cg_{t}|^{q(\cdot )}\big\|_{\frac{p(\cdot )}{q(\cdot )}}\leq
\int_{\alpha t}^{\beta t}\big\||f_{\tau }|^{q(\cdot )}\big\|_{\frac{p(\cdot )%
}{q(\cdot )}}\frac{d\tau }{\tau }+t,\quad t\in (0,1]
\end{equation*}%
holds for every sequence of functions $(f_{t})_{0<t\leq 1}$ and constant $%
m>n+c_{\log }(\frac{1}{q})$ such that the first term on right-hand side is
at most one, where the constant $c$ independent of $t$.\newline
$\mathrm{(ii)}$ The inequality 
\begin{equation*}
\big\|(g_{t})_{0<t\leq 1}\big\|_{\widetilde{\ell ^{q(\cdot )}(L^{p\left(
\cdot \right) })}}\leq c\big\|(f_{t})_{0<t\leq 1}\big\|_{\widetilde{\ell
^{q(\cdot )}(L^{p\left( \cdot \right) })}}
\end{equation*}%
holds for every sequence of functions $(f_{t})_{0<t\leq 1}$ and constant $%
m>n+c_{\log }(\frac{1}{q})$ such that the right-hand side is finite.
\end{lemma}

\begin{proof}
First let us prove (i). The claim can be reformulated as showing that%
\begin{equation*}
J:=\big\|c_{1}\text{ }\delta ^{-\frac{1}{q(\cdot )}}g_{t}\big\|_{p(\cdot
)}\leq 2^{1-\frac{1}{q^{-}}}+\ln \frac{\beta }{\alpha },\quad t\in (0,1],
\end{equation*}%
where $c_{1}>0$ and $\delta :=\int_{\alpha t}^{\beta t}\big\||f_{\tau
}|^{q(\cdot )}\big\|_{\frac{p(\cdot )}{q(\cdot )}}\frac{d\tau }{\tau }+t$.\
Applying Lemma \ref{DHR-lemma}, with an appropriate choice of $c_{1}$, we get%
\begin{eqnarray*}
J &\leq &\int_{\alpha t}^{\beta t}\big\|c_{1}\text{ }\delta ^{-\frac{1}{%
q(\cdot )}}(\eta _{\tau ,m}\ast f_{\tau })\big\|_{p(\cdot )}\frac{d\tau }{%
\tau } \\
&\leq &\int_{\alpha t}^{\beta t}\big\|\eta _{\tau ,m-c_{\log }(\frac{1}{q}%
)}\ast c_{1}\text{ }\delta ^{-\frac{1}{q(\cdot )}}|f_{\tau }|\big\|_{p(\cdot
)}\frac{d\tau }{\tau },\quad m>n+c_{\log }\big(\frac{1}{q}\big) \\
&\leq &\int_{\alpha t}^{\beta t}\big\|\delta ^{-\frac{1}{q(\cdot )}}f_{\tau }%
\big\|_{p(\cdot )}\frac{d\tau }{\tau },
\end{eqnarray*}%
since $\delta \in (t,1+t]$ and that the convolution with a radially
decreasing $L^{1}$-function is bounded in $L^{p(\cdot )}$, since $%
m>n+c_{\log }(\frac{1}{q})$. Write%
\begin{eqnarray*}
\int_{\alpha t}^{\beta t}\big\|\delta ^{-\frac{1}{q(\cdot )}}f_{\tau }\big\|%
_{p(\cdot )}\frac{d\tau }{\tau } &=&\int_{(\alpha t,\beta t]\cap B}\cdot
\cdot \cdot \frac{d\tau }{\tau }+\int_{(\alpha t,\beta t]\cap B^{c}}\cdot
\cdot \cdot \frac{d\tau }{\tau } \\
&=&J_{1,t}+J_{2,t},
\end{eqnarray*}%
where%
\begin{equation*}
B:=\big\{\tau >0:\big\||\delta ^{-\frac{1}{q(\cdot )}}f_{\tau }|^{q(\cdot )}%
\big\|_{\frac{p(\cdot )}{q(\cdot )}}\geq 1\big\}.
\end{equation*}%
By Lemma \ref{DZW}, 
\begin{equation*}
J_{1,t}\leq \int_{(\alpha t,\beta t]\cap B}\big\||\delta ^{-\frac{1}{q(\cdot
)}}f_{\tau }|^{q(\cdot )}\big\|_{\frac{p(\cdot )}{q(\cdot )}}^{\frac{1}{q^{-}%
}}\frac{d\tau }{\tau }\leq 2^{1-\frac{1}{q^{-}}}\delta ^{-1}\int_{\alpha
t}^{\beta t}\big\||f_{\tau }|^{q(\cdot )}\big\|_{\frac{p(\cdot )}{q(\cdot )}}%
\frac{d\tau }{\tau }\leq 2^{1-\frac{1}{q^{-}}}
\end{equation*}%
and%
\begin{equation*}
J_{2,t}\leq \int_{\alpha t}^{\beta t}\big\|\delta ^{-\frac{1}{q(\cdot )}%
}f_{\tau }\big\|_{p(\cdot )}\frac{d\tau }{\tau }\leq \int_{\alpha t}^{\beta
t}\frac{d\tau }{\tau }=\ln \frac{\beta }{\alpha }.
\end{equation*}%
Now we prove (ii). By the scaling argument, it suffices to consider the case%
\begin{equation*}
\big\|(f_{t})_{0<t\leq 1}\big\|_{\widetilde{\ell ^{q(\cdot )}(L^{p\left(
\cdot \right) })}}=1
\end{equation*}%
and show that the modular of $f$ on the left-hand side is bounded. In
particular, we show that%
\begin{equation*}
\int_{0}^{1}\big\||cg_{t}|^{q(\cdot )}\big\|_{\frac{p(\cdot )}{q(\cdot )}}%
\frac{dt}{t}\leq 2
\end{equation*}%
for some positive constant $c$. Applying Hardy inequality, see Lemma \ref%
{Hardy-inequality1} and the property (i) we obtain the desired result.
\end{proof}

\begin{lemma}
\label{Main-Lemma 2}Let $0<r<\infty \ $and $m>\max (n,\frac{n}{r})$. Let $\{%
\mathcal{F}\Phi ,\mathcal{F}\varphi \}$ be a resolution of unity: 
\begin{equation*}
\mathcal{F}\Phi (\xi )+\int_{0}^{1}\mathcal{F}\varphi (t\xi )\frac{dt}{t}%
=1,\quad \xi \in \mathbb{R}^{n}.
\end{equation*}%
$\mathrm{(i)}$ Let $\theta \in \mathcal{S}(\mathbb{R}^{n})$ be such that $%
\mathrm{supp}\,\mathcal{F}\theta \subset \{\xi \in \mathbb{R}^{n}:\left\vert
\xi \right\vert \leq 2\}$. There exists a constant $c>0$\ such that%
\begin{equation*}
|\theta \ast f|^{r}\leq c\text{ }\eta _{1,mr}\ast |\Phi \ast f|^{r}+c\text{ }%
\int_{1/4}^{1}\eta _{1,mr}\ast |\varphi _{\tau }\ast f|^{r}\frac{d\tau }{%
\tau }
\end{equation*}%
for any $f\in \mathcal{S}^{\prime }(\mathbb{R}^{n})$, where $\varphi _{\tau
}=\tau ^{-n}\varphi (\frac{\cdot }{\tau })$.\newline
$\mathrm{(ii)}$ Let $\omega \in \mathcal{S}(\mathbb{R}^{n})$ be such that $%
\mathrm{supp}\,\mathcal{F}\omega \subset \{\xi \in \mathbb{R}^{n}:\frac{1}{2}%
\leq \left\vert \xi \right\vert \leq 2\}$. There exists a constant $c>0$\
such that%
\begin{equation*}
|\omega _{t}\ast f|^{r}\leq c\text{ }\eta _{1,mr}\ast |\Phi \ast
f|^{r}+c\int_{t/4}^{\min (1,4t)}\eta _{\tau ,mr}\ast |\varphi _{\tau }\ast
f|^{r}\frac{d\tau }{\tau }
\end{equation*}%
for any $f\in \mathcal{S}^{\prime }(\mathbb{R}^{n})$ and any $0<t\leq 1$,
where $\omega _{t}=t^{-n}\omega (\frac{\cdot }{t})$.
\end{lemma}

\begin{proof}
We split the proof into two steps. First the case $1\leq r<\infty $ follows
by the H\"{o}lder inequality. \newline
\textit{Step 1.} Proof of (i).\ Since $\{\mathcal{F}\Phi ,\mathcal{F}\varphi
\}$ is a resolutions of unity, it follows that 
\begin{equation*}
\theta \ast f=\Phi \ast \theta \ast f+\int_{1/4}^{1}\theta \ast \varphi
_{\tau }\ast f\frac{d\tau }{\tau }.
\end{equation*}%
First recall the elementary inequality%
\begin{equation*}
d^{n}\eta _{d,m}(y-z)\leq d^{2n}\eta _{d,-m}(y-x)\eta _{d,m}(x-z),\quad
d>0,x,y,z\in \mathbb{R}^{n},
\end{equation*}%
which together with Lemma \ref{r-trick} implies that%
\begin{eqnarray*}
|\Phi \ast \theta \ast f(y)|^{r} &\lesssim &\eta _{1,mr}\ast |\Phi \ast
f|^{r}(y) \\
&=&c\int_{\mathbb{R}^{n}}\eta _{1,mr}(y-z)|\Phi \ast f(z)|^{r}dz \\
&\lesssim &\eta _{1,-mr}(y-x)\eta _{1,mr}\ast |\Phi \ast f|^{r}(x)
\end{eqnarray*}%
for any $x\in \mathbb{R}^{n}$ and any $m>\frac{n}{r}$. Furthermore, 
\begin{eqnarray*}
|\Phi \ast \theta \ast f(y)| &\leq &\int_{\mathbb{R}^{n}}\eta
_{1,N}(y-z)|\theta \ast f(z)|dz \\
&\leq &\eta _{1,-m}(y-x)\theta _{1}^{\ast ,m}f(x)\int_{\mathbb{R}^{n}}\eta
_{1,N-m}(y-z)dz \\
&\lesssim &\eta _{1,-m}(y-x)\theta _{1}^{\ast ,m}f(x)
\end{eqnarray*}%
for any $N>m+n$, where%
\begin{equation*}
\theta _{1}^{\ast ,m}f(x)=\sup_{y\in \mathbb{R}^{n}}\frac{|\theta \ast f(y)|%
}{(1+|y-x|)^{m}},\quad x\in \mathbb{R}^{n}.
\end{equation*}%
Therefore,%
\begin{equation*}
|\Phi \ast \theta \ast f(y)|\lesssim \eta _{1,-m}(y-x)(\theta _{1}^{\ast
,m}f(x))^{1-r}\eta _{1,mr}\ast |\Phi \ast f|^{r}(x)
\end{equation*}%
for any $x\in \mathbb{R}^{n}$ and any $m>n$. Again from Lemma \ref{r-trick}
we conclude%
\begin{eqnarray*}
|\theta \ast \varphi _{\tau }\ast f(y)|^{r} &\lesssim &\eta _{1,mr}\ast
|\varphi _{\tau }\ast f|^{r}(y) \\
&\lesssim &(1+|y-x|)^{mr}\eta _{1,mr}\ast |\varphi _{\tau }\ast f|^{r}(x)
\end{eqnarray*}%
and%
\begin{eqnarray*}
|\theta \ast \varphi _{\tau }\ast f(y)| &\lesssim &\int_{\mathbb{R}^{n}}\eta
_{\tau ,N}(y-z)|\theta \ast f(z)|dz,\quad \frac{1}{4}\leq \tau \leq 1 \\
&\lesssim &(1+|y-x|)^{m}\theta _{1}^{\ast ,m}f(x)
\end{eqnarray*}%
for any $x\in \mathbb{R}^{n}$, any $m>n$ and any $N>m+n$. Consequently%
\begin{equation}
\theta _{1}^{\ast ,m}f(x)\leq c(\theta _{1}^{\ast ,m}f(x))^{1-r}\Big(\eta
_{1,mr}\ast |\Phi \ast f|^{r}(x)+\int_{1/4}^{1}\eta _{1,mr}\ast |\varphi
_{\tau }\ast f|^{r}(x)\frac{d\tau }{\tau }\Big),  \label{est-theta1}
\end{equation}%
which implies that%
\begin{equation}
|\theta \ast f(x)|^{r}\leq c\text{ }\eta _{1,mr}\ast |\Phi \ast
f|^{r}(x)+c\int_{1/4}^{1}\eta _{1,mr}\ast |\varphi _{\tau }\ast f|^{r}(x)%
\frac{d\tau }{\tau }  \label{est-theta}
\end{equation}%
when $\theta _{1}^{\ast ,m}f(x)<\infty $, which is true if $m\geq \frac{n}{r}%
+N_{0}$ (order of distribution). We will use the Str\"{o}mberg and
Torchinsky idea \cite{ST89}. Observe that the right-hand side of %
\eqref{est-theta} decreases as $m$ increases. Therefore, we have %
\eqref{est-theta} for all $m>\frac{n}{r}$ but with $c=c(f)$ depending on $f$%
. We can easily check that if the right-hand side of \eqref{est-theta}, with 
$c=c(f)$, is finite imply that $\theta _{1}^{\ast ,m}f(x)<\infty $,
otherwise, there is nothing to prove. Returning to \eqref{est-theta1} and
having in mind that now $\theta _{1}^{\ast ,m}f(x)<\infty $, we obtain the
desired estimate$\mathrm{\ }$\eqref{est-theta}$\mathrm{.}$

\noindent \textit{Step 2.} Proof of (ii). We have%
\begin{equation*}
\omega _{t}\ast f=\int_{t/4}^{\min (1,4t)}\omega _{t}\ast \varphi _{\tau
}\ast f\frac{d\tau }{\tau }+\left\{ 
\begin{array}{ccc}
0, & \text{if} & 0<t<\frac{1}{4}; \\ 
\omega _{t}\ast \Phi \ast f, & \text{if} & \frac{1}{4}\leq t\leq 1.%
\end{array}%
\right.
\end{equation*}%
Let%
\begin{equation*}
g_{t}(y):=\int_{t/4}^{\min (1,4t)}\omega _{t}\ast \varphi _{\tau }\ast f(y)%
\frac{d\tau }{\tau },\quad y\in \mathbb{R}^{n},0<t\leq 1.
\end{equation*}%
It follows from Lemma \ref{r-trick} that%
\begin{eqnarray*}
|\omega _{t}\ast \varphi _{\tau }\ast f(y)|^{r} &\lesssim &\eta _{t,mr}\ast
|\varphi _{\tau }\ast f|^{r}(y) \\
&\lesssim &\eta _{\tau ,mr}\ast |\varphi _{\tau }\ast f|^{r}(y) \\
&=&c\int_{\mathbb{R}^{n}}\eta _{\tau ,mr}(y-z)|\varphi _{\tau }\ast
f(z)|^{r}dz \\
&\lesssim &(1+\tau ^{-1}|y-x|)^{mr}\eta _{\tau ,mr}\ast |\varphi _{\tau
}\ast f|^{r}(x)
\end{eqnarray*}%
and%
\begin{eqnarray*}
|\omega _{t}\ast \varphi _{\tau }\ast f(y)| &\lesssim &\int_{\mathbb{R}%
^{n}}\eta _{\tau ,N}(y-z)|\omega _{t}\ast f(z)|dz \\
&\lesssim &\omega _{t,m}^{\ast }f(y)\int_{\mathbb{R}^{n}}\eta _{\tau
,N}(y-z)(1+t^{-1}|y-z|)^{m}dz \\
&\lesssim &\omega _{t}^{\ast ,m}f(y) \\
&\lesssim &(1+t^{-1}|y-x|)^{m}\omega _{t}^{\ast ,m}f(x)
\end{eqnarray*}%
for any $x,y\in \mathbb{R}^{n}$, any $t/4\leq \tau \leq \min (1,4t),0<t\leq
1 $ and any $N>m+n$, where 
\begin{equation*}
\omega _{t}^{\ast ,m}f(x)=\sup_{y\in \mathbb{R}^{n}}\frac{|\omega _{t}\ast
f(y)|}{(1+t^{-1}|y-x|)^{m}},\quad x,y\in \mathbb{R}^{n},0<t\leq 1.
\end{equation*}%
Therefore, $|g_{t}(y)|$ can be estimated from above by%
\begin{eqnarray*}
&&c(\omega _{t}^{\ast ,m}f(x))^{1-r}(1+t^{-1}|y-x|)^{m(1-r)} \\
&&\times \int_{t/4}^{\min (1,4t)}(1+\tau ^{-1}|y-x|)^{mr}\eta _{\tau
,mr}\ast |\varphi _{\tau }\ast f|^{r}(x)\frac{d\tau }{\tau } \\
&\lesssim &(1+t^{-1}|y-x|)^{m}(\omega _{t}^{\ast
,m}f(x))^{1-r}\int_{t/4}^{\min (1,4t)}\eta _{\tau ,mr}\ast |\varphi _{\tau
}\ast f|^{r}(x)\frac{d\tau }{\tau },
\end{eqnarray*}%
if $0<t\leq 1$. Now if $\frac{1}{4}\leq t\leq 1$, we easily obtain

\begin{eqnarray*}
|\omega _{t}\ast \Phi \ast f(y)| &=&|\omega _{t}\ast \Phi \ast
f(y)|^{1-r}|\omega _{t}\ast \Phi \ast f(y)|^{r} \\
&\lesssim &(1+t^{-1}|y-x|)^{m(1-r)}(\omega _{t}^{\ast ,m}f(x))^{1-r}\eta
_{1,mr}\ast |\Phi \ast f|^{r}(y) \\
&\lesssim &(1+t^{-1}|y-x|)^{m}(\omega _{t}^{\ast ,m}f(x))^{1-r}\eta
_{1,mr}\ast |\Phi \ast f|^{r}(x),
\end{eqnarray*}%
which yields that%
\begin{equation*}
\sup_{y\in \mathbb{R}^{n}}\frac{|\omega _{t}\ast \Phi \ast f(y)|}{%
(1+t^{-1}|y-x|)^{m}}\lesssim (\omega _{t}^{\ast ,m}f(x))^{1-r}\eta
_{1,mr}\ast |\Phi \ast f|^{r}(x).
\end{equation*}%
Consequently%
\begin{equation*}
|\omega _{t}\ast f(x)|^{r}\lesssim \left( \omega _{t}^{\ast ,m}f(x)\right)
^{r}\lesssim \eta _{1,mr}\ast |\Phi \ast f|^{r}(x)+\int_{t/4}^{\min
(1,4t)}\eta _{\tau ,mr}\ast |\varphi _{\tau }\ast f|^{r}(x)\frac{d\tau }{%
\tau }.
\end{equation*}%
when $\omega _{t}^{\ast ,m}f(x)<\infty ,0<t\leq 1$ and $x\in \mathbb{R}^{n}$%
. Using a combination of the arguments used in (i), we arrive at the desired
estimate. The proof is complete.
\end{proof}

The following lemma is from \cite[Lemma 1]{Ry01}.

\begin{lemma}
\label{Ry-Lemma1} Let $\varrho ,\mu \in \mathcal{S}(\mathbb{R}^{n})$, and $%
M\geq -1$ an integer such that 
\begin{equation*}
\int_{\mathbb{R}^{n}}x^{\alpha }\mu (x)dx=0
\end{equation*}%
for all $\left\vert \alpha \right\vert \leq M$. Then for any $N>0$, there is
a constant $c(N)>0$ such that 
\begin{equation*}
\sup_{z\in \mathbb{R}^{n}}|t^{-n}\mu (t^{-1}\cdot )\ast \varrho
(z)|(1+\left\vert z\right\vert )^{N}\leq c(N)\text{ }t^{M+1},\quad 0<t\leq 1.
\end{equation*}
\end{lemma}

\section{Variable\ Besov\ spaces}

In this section we\ present the definition of\ Besov spaces of variable
smoothness and integrability, and prove the basic properties in analogy to
the case of fixed exponents. Select a pair of Schwartz functions $\Phi $ and 
$\varphi $ satisfying 
\begin{equation}
\mathrm{supp}\,\mathcal{F}\Phi \subset \{x\in \mathbb{R}^{n}:\left\vert
x\right\vert \leq 2\},\quad \mathrm{supp}\,\mathcal{F}\varphi \subset \big\{%
x\in \mathbb{R}^{n}:\frac{1}{2}\leq \left\vert x\right\vert \leq 2\big\}
\label{Ass1}
\end{equation}%
and 
\begin{equation}
\mathcal{F}\Phi (\xi )+\int_{0}^{1}\mathcal{F}\varphi (t\xi )\frac{dt}{t}%
=1,\quad \xi \in \mathbb{R}^{n}.  \label{Ass2}
\end{equation}%
Such a resolution \eqref{Ass1} and \eqref{Ass2} of unity can be constructed
as follows. Let $\mu \in \mathcal{S}(\mathbb{R}^{n})$\ be such that $%
\left\vert \mathcal{F}\mu (\xi )\right\vert >0$\ for $1/2<\left\vert \xi
\right\vert <2$. There exists $\eta \in \mathcal{S}(\mathbb{R}^{n})$\ with 
\begin{equation*}
\mathrm{supp}\,\mathcal{F}\eta \subset \big\{x\in \mathbb{R}^{n}:\frac{1}{2}%
<\left\vert x\right\vert <2\big\}
\end{equation*}%
such that 
\begin{equation*}
\int_{0}^{\infty }\mathcal{F}\mu (t\xi )\,\mathcal{F}\eta (t\xi )\frac{dt}{t}%
=1,\quad \xi \neq 0,
\end{equation*}%
see \cite{CaTor75}, \cite{Hei74} and \cite{JaTa81}. We set $\mathcal{F}%
\varphi =\mathcal{F}\mu \,\mathcal{F}\eta $ and 
\begin{equation*}
\mathcal{F}\Phi (\xi )=\left\{ 
\begin{array}{ccc}
\displaystyle{\int_{1}^{\infty }}\mathcal{F}\varphi (t\xi )\,\dfrac{dt}{t} & 
\text{if} & \xi \neq 0, \\ 
1 & \text{if} & \xi =0.%
\end{array}%
\right.
\end{equation*}%
Then $\mathcal{F}\Phi \in \mathcal{S}(\mathbb{R}^{n})$, and as $\mathcal{F}%
\eta $ is supported in $\{x\in \mathbb{R}^{n}:\frac{1}{2}\leq \left\vert
x\right\vert \leq 2\}$, we see that supp $\mathcal{F}\Phi \subset \{x\in 
\mathbb{R}^{n}:\left\vert x\right\vert \leq 2\}$.\newline

Now we define the spaces under consideration.

\begin{definition}
\label{B-F-def}Let $\alpha :\mathbb{R}^{n}\rightarrow \mathbb{R}$\ and\ $%
p,q\in \mathcal{P}_{0}(\mathbb{R}^{n})$. Let $\{\mathcal{F}\Phi ,\mathcal{F}%
\varphi \}$ be a resolution of unity and we put $\varphi _{t}=t^{-n}\varphi (%
\frac{\cdot }{t})$, $0<t\leq 1$. The Besov space $\mathfrak{B}_{p(\cdot
),q(\cdot )}^{\alpha (\cdot )}$\ is the collection of all $f\in \mathcal{S}%
^{\prime }(\mathbb{R}^{n})$\ such that%
\begin{equation*}
\big\|f\big\|_{\mathfrak{B}_{p(\cdot ),q(\cdot )}^{\alpha (\cdot )}}^{\Phi
,\varphi }:=\big\|\Phi \ast f\big\|_{p(\cdot )}+\big\|(t^{-\alpha (\cdot
)}\varphi _{t}\ast f)_{0<t\leq 1}\big\|_{\widetilde{\ell ^{q(\cdot
)}(L^{p\left( \cdot \right) })}}<\infty .
\end{equation*}
\end{definition}

When $q=\infty $,\ the Besov space $\mathfrak{B}_{p(\cdot ),\infty }^{\alpha
(\cdot )}$\ consist of all distributions $f\in \mathcal{S}^{\prime }(\mathbb{%
R}^{n})$\ such that 
\begin{equation*}
\big\|f\big\|_{\mathfrak{B}_{p(\cdot ),\infty }^{\alpha (\cdot )}}^{\Phi
,\varphi }:=\big\|\Phi \ast f\big\|_{p(\cdot )}+\sup_{t\in (0,1]}\big\|%
t^{-\alpha (\cdot )}(\varphi _{t}\ast f)\big\|_{p(\cdot )}<\infty .
\end{equation*}%
One recognizes immediately that $\mathfrak{B}_{p(\cdot ),q(\cdot )}^{\alpha
(\cdot )}$ is a quasi-normed space and if $\alpha $, $p$ and $q$ are
constants, then 
\begin{equation*}
\mathfrak{B}_{p(\cdot ),q(\cdot )}^{\alpha (\cdot )}=B_{p,q}^{\alpha },
\end{equation*}%
where $B_{p,q}^{\alpha }$ is the usual Besov spaces.

Now, we are ready to show that the definition of these\ function spaces is
independent of the chosen resolution $\{\mathcal{F}\Phi ,\mathcal{F}\varphi
\}$ of unity. This justifies our omission of the subscript $\Phi $ and $%
\varphi $ in the sequel.

\begin{theorem}
\label{Independent}Let $\{\mathcal{F}\Phi ,\mathcal{F}\varphi \}$ and $%
\left\{ \mathcal{F}\Psi ,\mathcal{F}\psi \right\} $ be two resolutions of
unity. Let $\alpha :\mathbb{R}^{n}\rightarrow \mathbb{R}$\ and\ $p,q\in 
\mathcal{P}_{0}(\mathbb{R}^{n})$. Assume that\ $p\in \mathcal{P}_{0}^{\log
}\left( \mathbb{R}^{n}\right) $ and $\alpha ,\frac{1}{q}\in C_{\mathrm{loc}%
}^{\log }(\mathbb{R}^{n})$. Then 
\begin{equation*}
\big\|f\big\|_{\mathfrak{B}_{p(\cdot ),q(\cdot )}^{\alpha (\cdot )}}^{\Phi
,\varphi }\approx \big\|f\big\|_{\mathfrak{B}_{p(\cdot ),q(\cdot )}^{\alpha
(\cdot )}}^{\Psi ,\psi }.
\end{equation*}
\end{theorem}

\begin{proof}
It is sufficient to show that there exists a constant $c>0$ such that for
all $f\in \mathfrak{B}_{p(\cdot ),q(\cdot )}^{\alpha (\cdot )}$ we have 
\begin{equation*}
\big\|f\big\|_{\mathfrak{B}_{p(\cdot ),q(\cdot )}^{\alpha (\cdot )}}^{\Phi
,\varphi }\lesssim \big\|f\big\|_{\mathfrak{B}_{p(\cdot ),q(\cdot )}^{\alpha
(\cdot )}}^{\Psi ,\psi }.
\end{equation*}%
In view of Lemma \ref{Main-Lemma 2} the problem can be reduced to the case
of $p\in \mathcal{P}^{\log }\left( \mathbb{R}^{n}\right) $ and $q\in 
\mathcal{P}(\mathbb{R}^{n})$ with $\frac{1}{q}\in C_{\mathrm{loc}}^{\log }(%
\mathbb{R}^{n})$. By the scaling argument, it suffices to consider the case $%
\big\|f\big\|_{\mathfrak{B}_{p(\cdot ),q(\cdot )}^{\alpha (\cdot )}}^{\Psi
,\psi }=1$ and show that%
\begin{equation*}
\big\|\Phi \ast f\big\|_{p(\cdot )}\lesssim 1
\end{equation*}%
and%
\begin{equation*}
\int_{0}^{1}\big\||c\text{ }t^{-\alpha (\cdot )}(\varphi _{t}\ast
f)|^{q(\cdot )}\big\|_{\frac{p(\cdot )}{q(\cdot )}}\frac{dt}{t}\leq 1
\end{equation*}%
for some positive constant $c$. Interchanging the roles of $\left( \Psi
,\psi \right) $ and $\left( \Phi ,\varphi \right) $ we obtain the desired
result. We have%
\begin{equation*}
\mathcal{F}\Phi (\xi )=\mathcal{F}\Phi (\xi )\mathcal{F}\Psi (\xi
)+\int_{1/4}^{1}\mathcal{F}\Phi (\xi )\mathcal{F}\psi (\tau \xi )\frac{d\tau 
}{\tau }
\end{equation*}%
and%
\begin{equation*}
\mathcal{F}\varphi (t\xi )=\int_{t/4}^{\min (1,4t)}\mathcal{F}\varphi (t\xi )%
\mathcal{F}\psi (\tau \xi )\frac{d\tau }{\tau }+\left\{ 
\begin{array}{ccc}
0, & \text{if} & 0<t<\frac{1}{4}; \\ 
\mathcal{F}\varphi (t\xi )\mathcal{F}\Psi (\xi ), & \text{if} & \frac{1}{4}%
\leq t\leq 1%
\end{array}%
\right.
\end{equation*}%
for any $\xi \in 
%TCIMACRO{\U{211d} }%
%BeginExpansion
\mathbb{R}
%EndExpansion
^{n}$. Then we see that 
\begin{equation*}
\Phi \ast f=\Phi \ast \Psi \ast f+\int_{1/4}^{1}\Phi \ast \psi _{\tau }\ast f%
\frac{d\tau }{\tau }
\end{equation*}%
and%
\begin{equation*}
\varphi _{t}\ast f=\int_{t/4}^{\min (1,4t)}\varphi _{t}\ast \psi _{\tau
}\ast f\frac{d\tau }{\tau }+\left\{ 
\begin{array}{ccc}
0, & \text{if} & 0<t<\frac{1}{4}; \\ 
\varphi _{t}\ast \Psi \ast f, & \text{if} & \frac{1}{4}\leq t\leq 1.%
\end{array}%
\right.
\end{equation*}%
First observe that%
\begin{eqnarray*}
|\Phi \ast \psi _{\tau }\ast f| &\lesssim &|\eta _{0,m}\ast \psi _{\tau
}\ast f| \\
&\lesssim &\eta _{0,m}\ast \tau ^{-\alpha (\cdot )}|\psi _{\tau }\ast
f|,\quad \frac{1}{4}\leq \tau \leq 1,m>n
\end{eqnarray*}%
and%
\begin{equation*}
|\Phi \ast \Psi \ast f|\lesssim \eta _{0,m}\ast |\Psi \ast f|,\quad m>n.
\end{equation*}%
Therefore,%
\begin{eqnarray*}
|\Phi \ast f| &\leq &\eta _{0,m}\ast |\Psi \ast f|+\int_{1/4}^{1}\eta
_{0,m}\ast \tau ^{-\alpha (\cdot )}|\psi _{\tau }\ast f|\frac{d\tau }{\tau }
\\
&=&\eta _{0,m}\ast |\Psi \ast f|+g.
\end{eqnarray*}%
Since\ $p\in \mathcal{P}^{\log }\left( 
%TCIMACRO{\U{211d} }%
%BeginExpansion
\mathbb{R}
%EndExpansion
^{n}\right) $ and the convolution with a radially decreasing $L^{1}$%
-function is bounded on $L^{p(\cdot )}$:%
\begin{equation*}
\big\|\eta _{0,m}\ast |\Psi \ast f|\big\|_{p(\cdot )}\lesssim \big\|\Psi
\ast f\big\|_{p(\cdot )}\leq 1.
\end{equation*}%
Now, for some suitable positive constant $c_{1}$, 
\begin{equation*}
\big\|c_{1}g\big\|_{p(\cdot )}\leq 1
\end{equation*}%
if and only if%
\begin{equation*}
\big\||c_{1}g|^{q(\cdot )}\big\|_{\frac{p(\cdot )}{q(\cdot )}}\leq 1,
\end{equation*}%
which follows by Lemma \ref{Al-Ha-Lemma 1}/(i). Therefore,%
\begin{equation*}
\big\|\Phi \ast f\big\|_{p(\cdot )}\lesssim 1.
\end{equation*}%
Using the fact that the convolution with a radially decreasing $L^{1}$%
-function is bounded in $L^{p(\cdot )}$, we obtain%
\begin{equation*}
\big\||c\varphi _{t}\ast \Psi \ast f|^{q(\cdot )}\big\|_{\frac{p(\cdot )}{%
q(\cdot )}}\leq 1,
\end{equation*}%
with an appropriate choice of $c$ and any\ $t\in (0,1]$. Observe that%
\begin{equation*}
|\varphi _{t}\ast f|\lesssim \int_{t/4}^{4t}\eta _{\tau ,m}\ast |\psi _{\tau
}\ast f|\frac{d\tau }{\tau },\quad m>n+c_{\log }(\frac{1}{q}),t\in (0,\frac{1%
}{4}].
\end{equation*}%
Applying\ again Lemma \ref{Al-Ha-Lemma 1}/(i), we find that%
\begin{equation*}
\int_{0}^{\frac{1}{4}}\big\||ct^{-\alpha (\cdot )}(\varphi _{t}\ast
f)|^{q(\cdot )}\big\|_{\frac{p(\cdot )}{q(\cdot )}}\frac{dt}{t}\leq 1
\end{equation*}%
for some suitable positive constant $c$. The proof of theorem is complete.
\end{proof}

Let $a>0$, $\alpha :\mathbb{R}^{n}\rightarrow \mathbb{R}$ and $f\in \mathcal{%
S}^{\prime }(\mathbb{R}^{n})$. Then we define the Peetre maximal function as
follows: 
\begin{equation*}
\varphi _{t}^{\ast ,a}t^{-\alpha (\cdot )}f(x):=\sup_{y\in \mathbb{R}^{n}}%
\frac{t^{-\alpha (y)}\left\vert \varphi _{t}\ast f(y)\right\vert }{\left(
1+t^{-1}\left\vert x-y\right\vert \right) ^{a}},\qquad t>0
\end{equation*}%
and%
\begin{equation*}
\Phi ^{\ast ,a}f(x):=\sup_{y\in \mathbb{R}^{n}}\frac{\left\vert \Phi \ast
f(y)\right\vert }{\left( 1+\left\vert x-y\right\vert \right) ^{a}}.
\end{equation*}%
We now present a fundamental characterization of the spaces under
consideration.

\begin{theorem}
\label{fun-char}Let $\alpha ,\frac{1}{q}\in C_{\mathrm{loc}}^{\log }(\mathbb{%
R}^{n})$, $p\in \mathcal{P}_{0}^{\log }\left( 
%TCIMACRO{\U{211d} }%
%BeginExpansion
\mathbb{R}
%EndExpansion
^{n}\right) $, $q^{-}\geq p^{-}$ and $a>\frac{n+c_{\log }(\frac{1}{q})}{p^{-}%
}+c_{\log }(\alpha )$. Then%
\begin{equation*}
\big\|f\big\|_{\mathfrak{B}_{p(\cdot ),q(\cdot )}^{\alpha (\cdot )}}^{\ast
}:=\big\|\Phi ^{\ast ,a}\big\|_{p(\cdot )}+\big\|(\varphi _{t}^{\ast
,a}t^{-\alpha (\cdot )}f)_{0<t\leq 1}\big\|_{\widetilde{\ell ^{q(\cdot
)}(L^{p\left( \cdot \right) })}}
\end{equation*}%
\textit{is an equivalent quasi-norm in }$\mathfrak{B}_{p(\cdot ),q(\cdot
)}^{\alpha (\cdot )}$.
\end{theorem}

\begin{proof}
It is easy to see that for any $f\in \mathcal{S}^{\prime }(\mathbb{R}^{n})$
with $\big\|f\big\|_{\mathfrak{B}_{p(\cdot ),q(\cdot )}^{\alpha (\cdot
)}}^{\ast }<\infty $ and any $x\in \mathbb{R}^{n}$ we have%
\begin{equation*}
t^{-\alpha (x)}\left\vert \varphi _{t}\ast f(x)\right\vert \leq \varphi
_{t}^{\ast ,a}t^{-\alpha (\cdot )}f(x)\text{.}
\end{equation*}%
This shows that $\big\|f\big\|_{\mathfrak{B}_{p(\cdot ),q(\cdot )}^{\alpha
(\cdot )}}\leq \big\|f\big\|_{\mathfrak{B}_{p(\cdot ),q(\cdot )}^{\alpha
(\cdot )}}^{\ast }$. We will prove that there is a constant $C>0$ such that
for every $f\in \mathfrak{B}_{p(\cdot ),q(\cdot )}^{\alpha (\cdot )}$%
\begin{equation}
\big\|f\big\|_{\mathfrak{B}_{p(\cdot ),q(\cdot )}^{\alpha (\cdot )}}^{\ast
}\leq C\big\|f\big\|_{\mathfrak{B}_{p(\cdot ),q(\cdot )}^{\alpha (\cdot )}}.
\label{estimate-B}
\end{equation}%
By Lemmas \ref{DHR-lemma} and \ref{r-trick}, the estimate 
\begin{eqnarray}
t^{-\alpha (y)}\left\vert \varphi _{t}\ast f(y)\right\vert  &\leq &C_{1}%
\text{ }t^{-\alpha (y)}\big(\eta _{t,\sigma p^{-}}\ast |\varphi _{t}\ast
f|^{p^{-}}(y)\big)^{1/p^{-}}  \notag \\
&\leq &C_{2}\text{ }\big(\eta _{t,(\sigma -c_{\log }(\alpha ))p^{-}}\ast
(t^{-\alpha (\cdot )}|\varphi _{t}\ast f|)^{p^{-}}(y)\big)^{1/p^{-}}
\label{esti-conv}
\end{eqnarray}%
is true for any $y\in \mathbb{R}^{n}$, $\sigma >\frac{n+c_{\log }(\frac{1}{q}%
)}{p^{-}}+c_{\log }(\alpha )$ and $t>0$. Now dividing both sides of %
\eqref{esti-conv} by $\left( 1+t^{-1}\left\vert x-y\right\vert \right) ^{a}$%
, in the right-hand side we use the inequality%
\begin{equation*}
\left( 1+t^{-1}\left\vert x-y\right\vert \right) ^{-a}\leq \left(
1+t^{-1}\left\vert x-z\right\vert \right) ^{-a}\left( 1+t^{-1}\left\vert
y-z\right\vert \right) ^{a},\quad x,y,z\in \mathbb{R}^{n},
\end{equation*}%
while in the left-hand side we take the supremum over $y\in \mathbb{R}^{n}$,
we find that for all $f\in \mathfrak{B}_{p(\cdot ),q(\cdot )}^{\alpha (\cdot
)}$ any $t>0$ and any $\sigma >\max \big(\frac{n+c_{\log }(\frac{1}{q})}{%
p^{-}}+c_{\log }(\alpha ),a+c_{\log }(\alpha )\big)$%
\begin{equation*}
\varphi _{t}^{\ast ,a}t^{-\alpha (\cdot )}f(x)\leq C_{2}\text{ }\big(\eta
_{t,ap^{-}}\ast (t^{-\alpha (\cdot )p^{-}}|\varphi _{t}\ast f|^{p^{-}})(x)%
\big)^{1/p^{-}},
\end{equation*}%
where $C_{2}>0$ is independent of $x,t$ and $f$. Assume that the right-hand
side of \eqref{estimate-B} is less than or equal $1$. We will prove that%
\begin{equation}
\big\|\Phi ^{\ast ,a}\big\|_{p(\cdot )}+\Big\|\Big(\big(\eta _{t,ap^{-}}\ast
(t^{-\alpha (\cdot )p^{-}}|\varphi _{t}\ast f|^{p^{-}})\big)^{1/p^{-}}\Big)%
_{0<t\leq 1}\Big\|_{\widetilde{\ell ^{q(\cdot )}(L^{p\left( \cdot \right) })}%
}\lesssim 1.  \label{estimate1}
\end{equation}%
Observe that the second quasi-norm of the left-hand side of \eqref{estimate1}
can be rewritten as%
\begin{equation}
\Big\|\Big(\eta _{t,ap^{-}}\ast (t^{-\alpha (\cdot )p^{-}}|\varphi _{t}\ast
f|^{p^{-}})\Big)_{0<t\leq 1}\Big\|_{\widetilde{\ell ^{\frac{q(\cdot )}{p^{-}}%
}(L^{\frac{p\left( \cdot \right) }{p^{-}}})}}^{\frac{1}{p^{-}}}.
\label{estimate2}
\end{equation}%
Let $0<t<\frac{1}{4}$. In view the proof of Lemma \ref{Main-Lemma 2}/(ii),
we obtain%
\begin{eqnarray*}
t^{-\alpha (\cdot )p^{-}}|\varphi _{t}\ast f|^{p^{-}} &\lesssim
&\int_{t/4}^{4t}\tau ^{-\alpha (\cdot )p^{-}}\eta _{\tau ,ap^{-}}\ast
|\varphi _{\tau }\ast f|^{p^{-}}\frac{d\tau }{\tau } \\
&\lesssim &\int_{t/4}^{4t}\eta _{\tau ,ap^{-}-c_{\log }(\alpha )p^{-}}\ast
\tau ^{-\alpha (\cdot )p^{-}}|\varphi _{\tau }\ast f|^{p^{-}}\frac{d\tau }{%
\tau },
\end{eqnarray*}%
by Lemma \ref{DHR-lemma}. Therefore,%
\begin{eqnarray*}
\eta _{t,ap^{-}}\ast (t^{-\alpha (\cdot )p^{-}}|\varphi _{t}\ast f|^{p^{-}})
&\lesssim &\int_{t/4}^{4t}\eta _{t,ap^{-}}\ast \eta _{\tau ,ap^{-}-c_{\log
}(\alpha )p^{-}}\ast \tau ^{-\alpha (\cdot )p^{-}}|\varphi _{\tau }\ast
f|^{p^{-}}\frac{d\tau }{\tau } \\
&\lesssim &\int_{t/4}^{4t}\eta _{\tau ,ap^{-}-c_{\log }(\alpha )p^{-}}\ast
\tau ^{-\alpha (\cdot )p^{-}}|\varphi _{\tau }\ast f|^{p^{-}}\frac{d\tau }{%
\tau },
\end{eqnarray*}%
by \cite[Lemma A.3]{DHR}. Applying Lemma \ref{Al-Ha-Lemma 1}, we deduce that %
\eqref{estimate2}, with $0<t<\frac{1}{4}$, is bounded by%
\begin{equation*}
\Big\|\big(t^{-\alpha (\cdot )p^{-}}|\varphi _{t}\ast f|^{p^{-}}\big)%
_{0<t\leq 1}\Big\|_{\widetilde{\ell ^{\frac{q(\cdot )}{p^{-}}}(L^{\frac{%
p\left( \cdot \right) }{p^{-}}})}}^{\frac{1}{p^{-}}}\lesssim 1.
\end{equation*}%
Now let $\frac{1}{4}\leq t\leq 1$. Again, by\ Lemma \ref{Main-Lemma 2}/(ii),
we get%
\begin{equation*}
|\varphi _{t}\ast f|^{p^{-}}\leq c\text{ }\eta _{1,ap^{-}}\ast |\Phi \ast
f|^{p^{-}}+c\int_{t/4}^{1}\eta _{\tau ,ap^{-}}\ast |\varphi _{\tau }\ast
f|^{p^{-}}\frac{d\tau }{\tau }.
\end{equation*}%
As above, we obtain%
\begin{eqnarray*}
&&\eta _{t,ap^{-}}\ast (t^{-\alpha (\cdot )p^{-}}|\varphi _{t}\ast
f|^{p^{-}}) \\
&\leq &c\text{ }\eta _{1,ap^{-}}\ast |\Phi \ast
f|^{p^{-}}+c\int_{t/4}^{1}\eta _{\tau ,ap^{-}-c_{\log }(\alpha )p^{-}}\ast
\tau ^{-\alpha (\cdot )p^{-}}|\varphi _{\tau }\ast f|^{p^{-}}\frac{d\tau }{%
\tau } \\
&=&c\text{ }\eta _{1,ap^{-}}\ast |\Phi \ast f|^{p^{-}}+h_{t}(x).
\end{eqnarray*}%
We need to prove that%
\begin{equation}
\Big\|\big(\eta _{1,ap^{-}}\ast |\Phi \ast f|^{p^{-}}\big)_{\frac{1}{4}\leq
t\leq 1}\Big\|_{\widetilde{\ell ^{\frac{q(\cdot )}{p^{-}}}(L^{\frac{p\left(
\cdot \right) }{p^{-}}})}}\lesssim 1\quad \text{and}\quad \big\|(h_{t})_{%
\frac{1}{4}\leq t\leq 1}\big\|_{\widetilde{\ell ^{\frac{q(\cdot )}{p^{-}}%
}(L^{\frac{p\left( \cdot \right) }{p^{-}}})}}\lesssim 1.  \label{estimate3}
\end{equation}%
Applying Lemma \ref{Al-Ha-Lemma 1}, we obtain the second estimate of %
\eqref{estimate3}\textrm{. }Let us prove the first one. This is equivalent to%
\begin{equation*}
\big\||\eta _{1,ap^{-}}\ast |\Phi \ast f|^{p^{-}}|^{\frac{q(\cdot )}{p^{-}}}%
\big\|_{\frac{p\left( \cdot \right) }{q(\cdot )}}\lesssim 1,
\end{equation*}%
which is equivalent to%
\begin{equation*}
\big\|\eta _{1,ap^{-}}\ast |\Phi \ast f|^{p^{-}}\big\|_{\frac{p\left( \cdot
\right) }{p^{-}}}\lesssim 1.
\end{equation*}%
Since\ $\frac{p\left( \cdot \right) }{p^{-}}\in \mathcal{P}^{\log }\left( 
%TCIMACRO{\U{211d} }%
%BeginExpansion
\mathbb{R}
%EndExpansion
^{n}\right) $ and the convolution with a radially decreasing $L^{1}$%
-function is bounded in $L^{p(\cdot )}$:%
\begin{equation*}
\big\|\eta _{1,ap^{-}}\ast |\Phi \ast f|^{p^{-}}\big\|_{\frac{p\left( \cdot
\right) }{p^{-}}}\lesssim \big\||\Phi \ast f|^{p^{-}}\big\|_{\frac{p\left(
\cdot \right) }{p^{-}}}=c\big\|\Phi \ast f\big\|_{p\left( \cdot \right)
}^{p^{-}}\lesssim 1.
\end{equation*}%
The estimate of $\big\|\Phi ^{\ast ,a}\big\|_{p(\cdot )}$ follows easily\
from the fact that%
\begin{equation*}
\big\|\Phi ^{\ast ,a}\big\|_{p(\cdot )}\lesssim \big\|\eta _{1,ap^{-}}\ast
|\Phi \ast f|^{p^{-}}\big\|_{\frac{p\left( \cdot \right) }{p^{-}}}^{\frac{1}{%
p^{-}}}\lesssim 1.
\end{equation*}%
The proof of Theorem \ref{fun-char} is complete.
\end{proof}

\section{Relation between $\mathfrak{B}_{p(\cdot ),q(\cdot )}^{\protect%
\alpha (\cdot )}$ and $B_{p(\cdot ),q(\cdot )}^{\protect\alpha (\cdot )}$}

In this section we present the coincidence between the above function spaces
and the variable Besov spaces of Almeida and H\"{a}st\"{o}, where to define
these function spaces we first need the concept of a smooth dyadic
resolution of unity. Let $\Psi $\ be a function\ in $\mathcal{S}(%
%TCIMACRO{\U{211d} }%
%BeginExpansion
\mathbb{R}
%EndExpansion
^{n})$\ satisfying $\Psi (x)=1$\ for\ $\left\vert x\right\vert \leq 1$\ and\ 
$\Psi (x)=0$\ for\ $\left\vert x\right\vert \geq 2$.\ We define $\psi _{0}$
and $\psi _{1}$ by $\mathcal{F}\psi _{0}(x)=\Psi (x)$, $\mathcal{F}\psi
_{1}(x)=\Psi (\frac{x}{2})-\Psi (x)$\ and 
\begin{equation*}
\mathcal{F}\psi _{v}(x)=\mathcal{F}\psi _{1}(2^{1-v}x)\quad \text{for}\quad
v=2,3,....
\end{equation*}%
Then $\{\mathcal{F}\psi _{v}\}_{v\in \mathbb{N}_{0}}$\ is a smooth dyadic
resolution of unity, $\sum_{v=0}^{\infty }\mathcal{F}\psi _{v}(x)=1$ for all 
$x\in \mathbb{R}^{n}$.\ Thus we obtain the Littlewood-Paley decomposition 
\begin{equation*}
f=\sum_{v=0}^{\infty }\psi _{v}\ast f
\end{equation*}%
for all $f\in \mathcal{S}^{\prime }(%
%TCIMACRO{\U{211d} }%
%BeginExpansion
\mathbb{R}
%EndExpansion
^{n})$\ $($convergence in $\mathcal{S}^{\prime }(%
%TCIMACRO{\U{211d} }%
%BeginExpansion
\mathbb{R}
%EndExpansion
^{n}))$.

We state the definition of the spaces $B_{p(\cdot ),q(\cdot )}^{s(\cdot )}$,
which introduced and investigated in \cite{AH}.

\begin{definition}
\label{11}\textit{Let }$\left\{ \mathcal{F}\psi _{v}\right\} _{v\in \mathbb{N%
}_{0}}$\textit{\ be a resolution of unity}, $s:\mathbb{R}^{n}\rightarrow 
\mathbb{R}$ and $p,q\in \mathcal{P}_{0}(\mathbb{R}^{n})$.\textit{\ The Besov
space }$B_{p(\cdot ),q(\cdot )}^{s(\cdot )}$\textit{\ consists of all
distributions }$f\in \mathcal{S}^{\prime }(%
%TCIMACRO{\U{211d} }%
%BeginExpansion
\mathbb{R}
%EndExpansion
^{n})$\textit{\ such that}%
\begin{equation*}
\big\|f\big\|_{B_{p(\cdot ),q(\cdot )}^{s(\cdot )}}:=\big\|(2^{vs(\cdot
)}\psi _{v}\ast f)_{v}\big\|_{\ell ^{q(\cdot )}(L^{p\left( \cdot \right)
})}<\infty .
\end{equation*}
\end{definition}

Taking $s\in \mathbb{R}$ and $q\in (0,\infty ]$ as constants we derive the
spaces $B_{p(\cdot ),q}^{s}$ studied by Xu in \cite{Xu08}. We refer the
reader to the recent papers \cite{AlCa151}, \cite{AlCa152}, \cite{D3} and 
\cite{KV122} for further details, historical remarks and more references on
these function spaces. For any $p,q\in \mathcal{P}_{0}^{\log }(\mathbb{R}%
^{n})$ and $s\in C_{\text{loc}}^{\log }$, the space $B_{p(\cdot ),q(\cdot
)}^{s(\cdot )}$ does not depend on the chosen smooth dyadic resolution of%
\textit{\ }unity $\{\mathcal{F}\psi _{v}\}_{v\in \mathbb{N}_{0}}$ (in the
sense of\ equivalent quasi-norms) and 
\begin{equation*}
\mathcal{S}(\mathbb{R}^{n})\hookrightarrow B_{p(\cdot ),q(\cdot )}^{s(\cdot
)}\hookrightarrow \mathcal{S}^{\prime }(\mathbb{R}^{n}).
\end{equation*}%
Moreover, if $p,q,s$ are constants, we re-obtain the usual Besov spaces $%
B_{p,q}^{s}$, studied in detail in \cite{T1} and\ \cite{T2}, see also \cite%
{S18}.

\begin{theorem}
\label{Besov-Al-Ha}Let $\alpha :\mathbb{R}^{n}\rightarrow \mathbb{R}$\ and\ $%
p,q\in \mathcal{P}_{0}(\mathbb{R}^{n})$. Assume that\ $p\in \mathcal{P}%
_{0}^{\log }\left( \mathbb{R}^{n}\right) $ and $\alpha ,\frac{1}{q}\in C_{%
\mathrm{loc}}^{\log }(\mathbb{R}^{n})$. \textit{Then}%
\begin{equation*}
\mathfrak{B}_{p(\cdot ),q(\cdot )}^{\alpha (\cdot )}=B_{p(\cdot ),q(\cdot
)}^{\alpha (\cdot )},
\end{equation*}%
in the sense of equivalent quasi-norms.
\end{theorem}

\begin{proof}
Step 1. We will prove that%
\begin{equation*}
\mathfrak{B}_{p(\cdot ),q(\cdot )}^{\alpha (\cdot )}\hookrightarrow
B_{p(\cdot ),q(\cdot )}^{\alpha (\cdot )}.
\end{equation*}%
From Lemma \ref{Main-Lemma 2} we only consider the case $p\in \mathcal{P}%
^{\log }\left( \mathbb{R}^{n}\right) $ and $q\in \mathcal{P}(\mathbb{R}^{n})$
with $\frac{1}{q}\in C_{\mathrm{loc}}^{\log }(\mathbb{R}^{n})$. Let $\{%
\mathcal{F}\Phi ,\mathcal{F}\varphi \}$ and $\left\{ \mathcal{F}\psi
_{j}\right\} _{j\in \mathbb{N}_{0}}$ be two resolutions of unity and let $%
f\in \mathfrak{B}_{p(\cdot ),q(\cdot )}^{\alpha (\cdot )}$ with%
\begin{equation*}
\big\|f\big\|_{\mathfrak{B}_{p(\cdot ),q(\cdot )}^{\alpha (\cdot )}}\leq 1.
\end{equation*}%
We have%
\begin{equation*}
\psi _{v}\ast f=\int_{2^{-v-2}}^{\min (1,2^{2-v})}\psi _{v}\ast \varphi
_{t}\ast f\frac{dt}{t}+\left\{ 
\begin{array}{ccc}
0, & \text{if} & v\geq 2; \\ 
\psi _{v}\ast \Phi \ast f, & \text{if} & v=0,1.%
\end{array}%
\right.
\end{equation*}%
Since the convolution with a radially decreasing $L^{1}$-function is bounded
in $L^{p(\cdot )}$, we obtain%
\begin{equation*}
\big\||c\text{ }\psi _{v}\ast \Phi \ast f|^{q(\cdot )}\big\|_{\frac{p(\cdot )%
}{q(\cdot )}}\leq 1,\quad v=0,1
\end{equation*}%
for some suitable positive constant $c$. Applying Lemma \ref{Al-Ha-Lemma 1},
we obtain%
\begin{equation*}
\big\||c_{1}\text{ }2^{v\alpha (\cdot )}\psi _{v}\ast f|^{q(\cdot )}\big\|_{%
\frac{p(\cdot )}{q(\cdot )}}\leq \int_{2^{-v-2}}^{\min (1,2^{2-v})}\big\|%
|t^{\alpha (\cdot )}\varphi _{t}\ast f|^{q(\cdot )}\big\|_{\frac{p(\cdot )}{%
q(\cdot )}}\frac{dt}{t}+2^{-v},\quad v\geq 2,
\end{equation*}%
with an appropriate choice of $c_{1}$. Taking the sum over $v\geq 2$, we
obtain $\big\|f\big\|_{B_{p(\cdot ),q(\cdot )}^{\alpha (\cdot )}}\lesssim 1.$

Step 2. We will prove that%
\begin{equation*}
B_{p(\cdot ),q(\cdot )}^{\alpha (\cdot )}\hookrightarrow \mathfrak{B}%
_{p(\cdot ),q(\cdot )}^{\alpha (\cdot )}.
\end{equation*}%
Let $\{\mathcal{F}\Phi ,\mathcal{F}\varphi \}$ and $\left\{ \mathcal{F}\psi
_{v}\right\} _{v\in \mathbb{N}_{0}}$ be two resolutions of unity and let $%
f\in B_{p(\cdot ),q(\cdot )}^{\alpha (\cdot )}$ with%
\begin{equation*}
\big\|f\big\|_{B_{p(\cdot ),q(\cdot )}^{\alpha (\cdot )}}\leq 1.
\end{equation*}%
We have%
\begin{eqnarray*}
\varphi _{t}\ast f &=&\sum_{v=0}^{\infty }\varphi _{t}\ast \psi _{v}\ast f \\
&=&\sum_{v=\lfloor \log _{2}(\frac{1}{2t})\rfloor }^{\lfloor \log _{2}(\frac{%
4}{t})\rfloor +1}\varphi _{t}\ast \psi _{v}\ast f+\left\{ 
\begin{array}{ccc}
0, & \text{if} & 0<t\leq \frac{1}{4}; \\ 
\psi _{0}\ast \Phi \ast f, & \text{if} & t>\frac{1}{4}%
\end{array}%
\right.
\end{eqnarray*}%
and%
\begin{equation*}
\Phi \ast f=\sum_{v=0}^{2}\Phi \ast \psi _{v}\ast f.
\end{equation*}%
Notice that if $v<0$ then we put $\psi _{v}\ast f=0$. Since the convolution
with a radially decreasing $L^{1}$-function is bounded in $L^{p(\cdot )}$,
we obtain%
\begin{equation*}
\big\||c\psi _{v}\ast \Phi \ast f|^{q(\cdot )}\big\|_{\frac{p(\cdot )}{%
q(\cdot )}}\leq 1,\quad v=0,1,2,
\end{equation*}%
which yields, 
\begin{equation*}
\big\|c|\Phi \ast f|^{q(\cdot )}\big\|_{\frac{p(\cdot )}{q(\cdot )}}\leq 1
\end{equation*}%
for some suitable positive constant $c$. Let $t\in (2^{-i},2^{-i+1}]$, $i\in 
\mathbb{N}$. We have%
\begin{eqnarray*}
t^{-\alpha (\cdot )}|\varphi _{t}\ast f| &\lesssim &\sum_{v=\lfloor \log
_{2}(\frac{1}{2t})\rfloor }^{\lfloor \log _{2}(\frac{4}{t})\rfloor
+1}t^{-\alpha (\cdot )}\eta _{t,m}\ast |\psi _{v}\ast f| \\
&\lesssim &\sum_{v=i-3}^{i-1}2^{(i-v)\alpha ^{-}}\eta _{v,m-c_{\log }(\alpha
)}\ast 2^{v\alpha (\cdot )}|\psi _{v}\ast f| \\
&\leq &c\sum_{j=-3}^{-1}\eta _{j+i,m-c_{\log }(\alpha )}\ast 2^{(j+i)\alpha
(\cdot )}|\psi _{j+i}\ast f|,
\end{eqnarray*}%
where $m>n+c_{\log }(\alpha )+c_{\log }(\frac{1}{q})$. Now observe that%
\begin{eqnarray*}
&&\int_{0}^{1}\big\||c\text{ }t^{-\alpha (\cdot )}\varphi _{t}\ast
f|^{q(\cdot )}\big\|_{\frac{p(\cdot )}{q(\cdot )}}\frac{dt}{t} \\
&=&\sum_{i=0}^{\infty }\int_{2^{-i}}^{2^{1-i}}\big\||t^{-\alpha (\cdot
)}\varphi _{t}\ast f|^{q(\cdot )}\big\|_{\frac{p(\cdot )}{q(\cdot )}}\frac{dt%
}{t} \\
&\leq &\sum_{i=0}^{\infty }\big\|\big(c\text{ }\sum_{j=-3}^{-1}\eta
_{j+i,m-c_{\log }(\alpha )}\ast 2^{(j+i)\alpha (\cdot )}|\psi _{j+i}\ast f|%
\big)^{q(\cdot )}\big\|_{\frac{p(\cdot )}{q(\cdot )}}
\end{eqnarray*}%
for some suitable positive constant $c$. The desired estimate follows by
Lemma \ref{Alm-Hastolemma1}. The proof is complete.
\end{proof}

In order to formulate the main result of this section, let us consider $%
k_{0},k\in \mathcal{S}(\mathbb{R}^{n})$ and $S\geq -1$ an integer such that
for an $\varepsilon >0$%
\begin{align}
\left\vert \mathcal{F}k_{0}(\xi )\right\vert & >0\text{\quad for\quad }%
\left\vert \xi \right\vert <2\varepsilon ,  \label{T-cond1} \\
\left\vert \mathcal{F}k(\xi )\right\vert & >0\text{\quad for\quad }\frac{%
\varepsilon }{2}<\left\vert \xi \right\vert <2\varepsilon   \label{T-cond2}
\end{align}%
and%
\begin{equation}
\int_{\mathbb{R}^{n}}x^{\alpha }k(x)dx=0\text{\quad for any\quad }\left\vert
\alpha \right\vert \leq S.  \label{moment-cond}
\end{equation}%
Here \eqref{T-cond1} and \eqref{T-cond2}\ are Tauberian conditions, while %
\eqref{moment-cond} states that moment conditions on $k$. We recall the
notation%
\begin{equation*}
k_{t}(x):=t^{-n}k(t^{-1}x)\text{\quad for}\quad t>0.
\end{equation*}%
For any $a>0$, $f\in \mathcal{S}^{\prime }(\mathbb{R}^{n})$ and $x\in 
\mathbb{R}^{n}$ we denote%
\begin{equation*}
k_{t}^{\ast ,a}t^{-\alpha \left( \cdot \right) }f(x):=\sup_{y\in \mathbb{R}%
^{n}}\frac{t^{-\alpha \left( y\right) }\left\vert k_{t}\ast f(y)\right\vert 
}{\left( 1+t^{-1}\left\vert x-y\right\vert \right) ^{a}},\quad j\in \mathbb{N%
}_{0}.
\end{equation*}%
We are now able to state the so called local mean characterization of $%
B_{p(\cdot ),q(\cdot )}^{\alpha (\cdot )}$ spaces, which is a more general
form of Theorem \ref{fun-char}.\vskip5pt

\begin{theorem}
\label{loc-mean-char}Let $\alpha ,\frac{1}{q}\in C_{\mathrm{loc}}^{\log }(%
\mathbb{R}^{n})$, $p\in \mathcal{P}_{0}^{\log }\left( 
%TCIMACRO{\U{211d} }%
%BeginExpansion
\mathbb{R}
%EndExpansion
^{n}\right) $, $a>\frac{n}{p^{-}}$ and $\alpha ^{+}<S+1$\textit{. Then}%
\begin{equation*}
\big\|f\big\|_{B_{p\left( \cdot \right) ,q\left( \cdot \right) }^{\alpha
\left( \cdot \right) }}^{\prime }:=\big\|k_{0}^{\ast ,a}f\big\|_{p(\cdot )}+%
\big\|(k_{t}^{\ast ,a}t^{-\alpha (\cdot )}f)_{0<t\leq 1}\big\|_{\widetilde{%
\ell ^{q(\cdot )}(L^{p\left( \cdot \right) })}}
\end{equation*}%
\textit{is an equivalent quasi-norm on }$B_{p(\cdot ),q(\cdot )}^{\alpha
(\cdot )}$.
\end{theorem}

\begin{proof}
The idea of the proof is from V. S. Rychkov \cite{Ry01}. The proof is
divided into three steps.

\noindent \textit{Step }1\textit{.} Let $\varepsilon >0$. Take any pair of
functions $\varphi _{0}$ and $\varphi \in \mathcal{S}(\mathbb{R}^{n})$ such
that%
\begin{align*}
|\mathcal{F}\varphi _{0}(\xi )|& >0\quad \text{for}\quad |\xi |<2\varepsilon
, \\
|\mathcal{F}\varphi (\xi )|& >0\quad \text{for}\quad \frac{\varepsilon }{2}%
<|\xi |<2\varepsilon .
\end{align*}%
We prove that there is a constant $c>0$ such that for any $f\in B_{p(\cdot
),q(\cdot )}^{\alpha (\cdot )}$%
\begin{equation}
\big\|f\big\|_{B_{p(\cdot ),q(\cdot )}^{\alpha (\cdot )}}^{\prime }\leq c%
\big\|\varphi _{0}^{\ast ,a}f\big\|_{p(\cdot )}+\Big\|\big(\varphi
_{j}^{\ast ,a}2^{j\alpha (\cdot )}f\big)_{j\geq 1}\Big\|_{\ell ^{q(\cdot
)}(L^{p(\cdot )})}.  \label{key-est1}
\end{equation}%
Let $\Lambda $, $\lambda \in \mathcal{S}(\mathbb{R}^{n})$ such that%
\begin{equation*}
\text{supp }\mathcal{F}\Lambda \subset \{\xi \in \mathbb{R}^{n}:|\xi
|<2\varepsilon \}\text{,\quad supp }\mathcal{F}\lambda \subset \{\xi \in 
\mathbb{R}^{n}:\varepsilon /2<|\xi |<2\varepsilon \}
\end{equation*}%
and%
\begin{equation*}
\mathcal{F}\Lambda (\xi )\mathcal{F}\varphi _{0}(\xi )+\sum_{j=1}^{\infty }%
\mathcal{F}\lambda (2^{-j}\xi )\mathcal{F}\varphi (2^{-j}\xi )=1,\quad \xi
\in \mathbb{R}^{n}.
\end{equation*}%
In particular, for any $f\in B_{p(\cdot ),q(\cdot )}^{\alpha (\cdot )}$ the
following identity is true:%
\begin{equation*}
f=\Lambda \ast \varphi _{0}\ast f+\sum_{j=1}^{\infty }\lambda _{j}\ast
\varphi _{j}\ast f,
\end{equation*}%
where%
\begin{equation*}
\varphi _{j}:=2^{jn}\varphi (2^{j}\cdot )\quad \text{and}\quad \lambda
_{j}:=2^{jn}\lambda (2^{j}\cdot ),\quad j\in \mathbb{N}.
\end{equation*}%
Hence we can write%
\begin{equation*}
k_{t}\ast f=k_{t}\ast \Lambda \ast \varphi _{0}\ast f+\sum_{j=1}^{\infty
}k_{t}\ast \lambda _{j}\ast \varphi _{j}\ast f,\quad t\in (0,1].
\end{equation*}%
Let $2^{-i}<t\leq 2^{1-i}$, $i\in \mathbb{N}_{0}$. First, let $j<i$. Writing
for any $z\in \mathbb{R}^{n}$%
\begin{equation*}
k_{t}\ast \lambda _{j}(z)=2^{jn}k_{2^{j}t}\ast \lambda (2^{j}z),
\end{equation*}%
we deduce from Lemma \ref{Ry-Lemma1} that for any $N>0$ there is a constant $%
c>0$ independent of $t$ and $j$ such that 
\begin{equation*}
|k_{t}\ast \lambda _{j}(z)|\leq c\text{ }\left( 2^{j}t\right) ^{S+1}\eta
_{j,N}(z),\quad z\in \mathbb{R}^{n}.
\end{equation*}%
This together with Lemma \ref{DHR-lemma} yield that%
\begin{equation*}
t^{-\alpha (y)}|k_{t}\ast \lambda _{j}\ast \varphi _{j}\ast f(y)|,
\end{equation*}%
can be estimated from above by 
\begin{equation*}
c2^{(j-i)(S+1-\alpha ^{+})}\varphi _{j}^{\ast ,a}2^{j\alpha (\cdot
)}f(y)\int_{\mathbb{R}^{n}}\eta _{j,N-c_{\log }(\alpha )-a}(y-z)dz\lesssim
2^{(j-i)(S+1-\alpha ^{+})}\varphi _{j}^{\ast ,a}2^{j\alpha (\cdot )}f(y)
\end{equation*}%
for any $N>n+a+c_{\log }(\alpha )$ any $y\in \mathbb{R}^{n}$ and any $j<i$.
Next, let $j\geq i$. Then, again by Lemma \ref{Ry-Lemma1}, we have for any $%
z\in \mathbb{R}^{n}$ and any $L>0$ 
\begin{equation*}
|k_{t}\ast \lambda _{j}(z)|=t^{-n}\big|k\ast \lambda _{\frac{1}{2^{j}t}}(%
\frac{z}{t})\big|\leq c\big(\frac{1}{2^{j}t}\big)^{M+1}\eta _{t,L}(z),
\end{equation*}%
where an integer $M\geq -1$ is taken arbitrarily large, since $D^{\beta }%
\mathcal{F}\lambda (0)=0$ for all $\beta $. Hence, again with Lemma \ref%
{DHR-lemma}, 
\begin{eqnarray*}
&&t^{-\alpha (y)}|k_{t}\ast \lambda _{j}\ast \varphi _{j}\ast f(y)| \\
&\leq &t^{-\alpha (y)}\int_{\mathbb{R}^{n}}|k_{t}\ast \lambda
_{j}(y-z)||\varphi _{j}\ast f(z)|dz \\
&\lesssim &2^{(i-j)(M+1+\alpha ^{-})-jn}\varphi _{j}^{\ast ,a}2^{j\alpha
(\cdot )}f(y)\int_{\mathbb{R}^{n}}\eta _{j,-c_{\log }(\alpha )-a}(y-z)\eta
_{i,L}(y-z)dz.
\end{eqnarray*}%
We have for any $j\geq i$%
\begin{equation*}
\left( 1+2^{j}\left\vert z\right\vert \right) ^{c_{\log }(\alpha )+a}\leq
2^{(j-i)(c_{\log }(\alpha )+a)}\left( 1+2^{i}\left\vert z\right\vert \right)
^{c_{\log }(\alpha )+a}.
\end{equation*}%
Then, by taking $L>n+a+c_{\log }(\alpha )$, 
\begin{equation*}
t^{-\alpha (y)}|k_{t}\ast \lambda _{j}\ast \varphi _{j}\ast f(y)|\lesssim
2^{(i-j)(M+1+\alpha ^{-}-c_{\log }(\alpha )-a)}\varphi _{j}^{\ast
,a}2^{j\alpha (\cdot )}f(y).
\end{equation*}%
Let us take $M>c_{\log }(\alpha )-\alpha ^{-}+2a$ to estimate the last
expression by%
\begin{equation*}
c\text{ }2^{(i-j)(a+1)}\varphi _{j}^{\ast ,a}2^{j\alpha (\cdot )}f(y),
\end{equation*}%
where $c>0$ is independent of $i,j$ and $f$. Using the fact that for any $%
z\in \mathbb{R}^{n}$ and any $N>0$ 
\begin{equation*}
|k_{t}\ast \Lambda (z)|\leq c\text{ }t^{S+1}\eta _{1,N}(z),
\end{equation*}%
we obtain by the similar arguments that for any $2^{-i}\leq t\leq 2^{-i+1}$, 
$i\in \mathbb{N}$%
\begin{equation*}
\sup_{y\in \mathbb{R}^{n}}\frac{t^{-\alpha (y)}|k_{t}\ast \Lambda \ast
\varphi _{0}\ast f(y)|}{(1+t^{-1}\left\vert x-y\right\vert )^{a}}\leq C\text{
}2^{-i(S+1-\alpha ^{+})}\varphi _{0}^{\ast ,a}f(x).
\end{equation*}%
Further, note that for all $x,y\in \mathbb{R}^{n}$ all $2^{-i}\leq t\leq
2^{1-i}$, $i\in \mathbb{N}$ and any $j\in \mathbb{N}_{0}$%
\begin{eqnarray*}
\varphi _{j}^{\ast ,a}2^{j\alpha (\cdot )}f(y) &\leq &\varphi _{j}^{\ast
,a}2^{j\alpha (\cdot )}f(x)(1+2^{j}\left\vert x-y\right\vert )^{a} \\
&\leq &\varphi _{j}^{\ast ,a}2^{j\alpha (\cdot )}f(x)\max
(1,2^{(j-i)a})(1+2^{i}\left\vert x-y\right\vert )^{a}.
\end{eqnarray*}%
Hence%
\begin{equation*}
\sup_{y\in \mathbb{R}^{n}}\frac{t^{-\alpha (y)}|k_{t}\ast \lambda _{j}\ast
\varphi _{j}\ast f(y)|}{(1+t^{-1}\left\vert x-y\right\vert )^{a}}\leq C\text{
}\varphi _{j}^{\ast ,a}2^{j\alpha (\cdot )}f(x)\times \left\{ 
\begin{array}{ccc}
2^{(j-i)(S+1-\alpha ^{+})} & \text{if} & j<i, \\ 
2^{i-j} & \text{if} & j\geq i.%
\end{array}%
\right. 
\end{equation*}%
Therefore for all $f\in B_{p(\cdot ),q(\cdot )}^{\alpha (\cdot )}$, any $%
x\in \mathbb{R}^{n}$ and any $2^{-i}\leq t\leq 2^{1-i}$, $i\in \mathbb{N}_{0}
$, we get 
\begin{eqnarray*}
&&k_{t}^{\ast ,a}t^{-\alpha (\cdot )}f(x) \\
&\lesssim &2^{-i(S+1-\alpha ^{+})}\varphi _{0}^{\ast
,a}f(x)+C\sum_{j=1}^{\infty }\min \Big(2^{(j-i)(S+1-\alpha ^{+})},2^{i-j}%
\Big)\varphi _{j}^{\ast ,a}2^{j\alpha (\cdot )}f(x) \\
&=&C\sum_{j=0}^{\infty }\min \Big(2^{(j-i)(S+1-\alpha ^{+})},2^{i-j}\Big)%
\varphi _{j}^{\ast ,a}2^{j\alpha (\cdot )}f(x) \\
&=&C\Psi _{i}(x).
\end{eqnarray*}%
Assume that the right hand side of \eqref{key-est1} is less than or equal
one. We have%
\begin{eqnarray*}
\int_{0}^{1}\big\||k_{t}^{\ast ,a}t^{-\alpha (\cdot )}f|^{q(\cdot )}\big\|_{%
\frac{p(\cdot )}{q(\cdot )}}\frac{dt}{t} &=&\sum_{i=0}^{\infty
}\int_{2^{-i}}^{2^{1-i}}\big\||k_{t}^{\ast ,a}t^{-\alpha (\cdot
)}f|^{q(\cdot )}\big\|_{\frac{p(\cdot )}{q(\cdot )}}\frac{dt}{t} \\
&\leq &\sum_{i=0}^{\infty }\big\||c\Psi _{i}|^{q(\cdot )}\big\|_{\frac{%
p(\cdot )}{q(\cdot )}}
\end{eqnarray*}%
for some positive constant $c$. The last term on the right hand side is less
than or equal one if and only if%
\begin{equation*}
\big\|\left( c_{1}\Psi _{i}\right) _{i}\big\|_{\ell ^{q(\cdot )}(L^{p(\cdot
)})}\leq 1
\end{equation*}%
for some suitable positive constant $c_{1}$, which follows by Lemma 8 of 
\cite{KV122} and the fact that $\alpha ^{+}<S+1$. Also we have for any $z\in 
\mathbb{R}^{n}$, any $N>0$ and any integer $M\geq -1$%
\begin{equation*}
\left\vert k_{0}\ast \lambda _{j}(z)\right\vert \leq c2^{-j(M+1)}\eta
_{j,N}(z)\quad \text{and}\quad \left\vert k_{0}\ast \Lambda (z)\right\vert
\leq c\text{ }\eta _{1,N}(z).
\end{equation*}%
As before, we get for any $x\in \mathbb{R}^{n}$%
\begin{equation}
k_{0}^{\ast ,a}f(x)\leq C\varphi _{0}^{\ast ,a}f(x)+C\sum_{j=1}^{\infty
}2^{-j}\varphi _{j}^{\ast ,a}2^{j\alpha (\cdot )}f(x).  \label{estk0}
\end{equation}%
In \eqref{estk0} taking the $L^{p(\cdot )}$-quasi-norm and using the
embedding $\ell ^{q(\cdot )}(L^{p(\cdot )})\hookrightarrow \ell ^{\infty
}(L^{p(\cdot )})$ we get \eqref{key-est1}.

\noindent \textit{Step} 2.\ Let $\left\{ \mathcal{F}\varphi _{j}\right\}
_{j\in \mathbb{N}_{0}}\subset \mathcal{S}(\mathbb{R}^{n})$ be such that 
\begin{equation*}
\mathrm{supp}\,\mathcal{F}\varphi \subset \big\{\xi \in \mathbb{R}%
^{n}:\varepsilon /2\leq \left\vert \xi \right\vert \leq 2\varepsilon \big\}
\end{equation*}%
and%
\begin{equation*}
\mathrm{supp}\,\mathcal{F}\varphi _{0}\subset \big\{\xi \in \mathbb{R}%
^{n}:\left\vert \xi \right\vert \leq 2\varepsilon \big\},\quad \varepsilon
>0,
\end{equation*}%
with $\varphi _{j}=2^{jn}\varphi (2^{j}\cdot ),j\in \mathbb{N}$. We will
prove that%
\begin{equation}
\big\|\varphi _{0}\ast f\big\|_{p(\cdot )}+\Big\|\big(2^{j\alpha (\cdot
)}(\varphi _{j}\ast f)\big)_{j\geq 1}\Big\|_{\ell ^{q(\cdot )}(L^{p(\cdot
)})}\lesssim \big\|f\big\|_{B_{p(\cdot ),q(\cdot )}^{\alpha (\cdot
)}}^{\prime }  \label{fin-est}
\end{equation}%
Let $\Lambda $, $\lambda \in \mathcal{S}(\mathbb{R}^{n})$ such that 
\begin{equation*}
\text{supp }\mathcal{F}\Lambda \subset \{\xi \in \mathbb{R}^{n}:\left\vert
\xi \right\vert <2\varepsilon \}\text{,\quad supp }\mathcal{F}\lambda
\subset \{\xi \in \mathbb{R}^{n}:\varepsilon /2<\left\vert \xi \right\vert
<2\varepsilon \},
\end{equation*}%
\begin{equation*}
\mathcal{F}\Lambda (\xi )\mathcal{F}k_{0}(\xi )+\int_{0}^{1}\mathcal{F}%
\lambda (\tau \xi )\mathcal{F}k(\tau \xi )\frac{d\tau }{\tau }=1,\quad \xi
\in \mathbb{R}^{n}.
\end{equation*}%
In particular, for any $f\in B_{p(\cdot ),q(\cdot )}^{\alpha (\cdot )}$ the
following identity is true:%
\begin{equation*}
f=\Lambda \ast k_{0}\ast f+\int_{0}^{1}\lambda _{\tau }\ast k_{\tau }\ast f%
\frac{d\tau }{\tau }.
\end{equation*}%
Hence we can write%
\begin{equation*}
\varphi _{j}\ast f=\int_{0}^{1}\varphi _{j}\ast \lambda _{\tau }\ast k_{\tau
}\ast f\frac{d\tau }{\tau }=\int_{2^{-j-2}}^{2^{-j+2}}\varphi _{j}\ast
\lambda _{\tau }\ast k_{\tau }\ast f\frac{d\tau }{\tau },\quad j\geq 2.
\end{equation*}%
Using the fact that%
\begin{equation*}
\max (|k_{\tau }\ast \lambda _{\tau }(z)|,|\varphi _{j}\ast \lambda _{\tau
}(z)|)\lesssim \eta _{j,N}(z),\quad z\in \mathbb{R}^{n},2^{-j-2}\leq \tau
\leq 2^{-j+2},j\in \mathbb{N}
\end{equation*}%
and Lemma \ref{DHR-lemma}, with $N>0$ large enough, we easily obtain%
\begin{equation*}
2^{j\alpha (y)}|\varphi _{j}\ast \lambda _{\tau }\ast k_{\tau }\ast
f(y)|\lesssim \min (k_{\tau }^{\ast ,a}\tau ^{-\alpha (\cdot )}f(y),\varphi
_{j}^{\ast ,a}2^{j\alpha (y)}f(y))
\end{equation*}%
for any $y\in \mathbb{R}^{n}$ and any $2^{-j+2}\leq \tau \leq 2^{-j-2},j\in 
\mathbb{N}$. Therefore%
\begin{equation*}
2^{j\alpha (y)}|\varphi _{j}\ast f(y)|\lesssim \big(\varphi _{j}^{\ast
,a}2^{j\alpha (\cdot )}f(y)\big)^{1-r}\int_{2^{-j-2}}^{2^{-j+2}}\left(
k_{\tau }^{\ast ,a}\tau ^{-\alpha (\cdot )}f(y)\right) ^{r}\frac{d\tau }{%
\tau },\quad 0<r<1,
\end{equation*}%
which yields that%
\begin{equation*}
\varphi _{j}^{\ast ,a}2^{j\alpha (\cdot )}f(x)\lesssim \big(\varphi
_{j}^{\ast ,a}2^{j\alpha (\cdot )}f(x)\big)^{1-r}\int_{2^{-j-2}}^{2^{-j+2}}%
\left( k_{\tau }^{\ast ,a}\tau ^{-\alpha (\cdot )}f(x)\right) ^{r}\frac{%
d\tau }{\tau }.
\end{equation*}%
This estimate gives%
\begin{equation*}
\big(\varphi _{j}^{\ast ,a}2^{j\alpha (\cdot )}f(x)\big)^{r}\lesssim
\int_{2^{-j-2}}^{2^{-j+2}}\left( k_{\tau }^{\ast ,a}\tau ^{-\alpha (\cdot
)}f(x)\right) ^{r}\frac{d\tau }{\tau }
\end{equation*}%
and%
\begin{equation}
2^{j\alpha (x)r}|\varphi _{j}\ast f(x)|^{r}\lesssim
\int_{2^{-j-2}}^{2^{-j+2}}\left( k_{\tau }^{\ast ,a}\tau ^{-\alpha (\cdot
)}f(x)\right) ^{r}\frac{d\tau }{\tau },\quad x\in \mathbb{R}^{n},
\label{main-est}
\end{equation}%
but if $\varphi _{j}^{\ast ,a}2^{j\alpha (\cdot )}f(x)<\infty $. Using a
combination of the arguments used in Lemma \ref{Main-Lemma 2}, we get %
\eqref{main-est} for all $0<r<1$, $a>0$ and all $f\in B_{p(\cdot ),q(\cdot
)}^{\alpha (\cdot )}$. Similarly we obtain%
\begin{equation*}
|\varphi _{j}\ast f(x)|^{r}\lesssim \left( k_{0}^{\ast ,a}f(x)\right)
^{r}+\int_{\frac{1}{8}}^{1}\left( k_{\tau }^{\ast ,a}\tau ^{-\alpha (\cdot
)}f(x)\right) ^{r}\frac{d\tau }{\tau },\quad j=0,1
\end{equation*}%
for any $0<r<1,a>0$ and any $f\in B_{p(\cdot ),q(\cdot )}^{\alpha (\cdot )}$%
. Let $\theta >0$ be such that $\max (1,\frac{(\frac{1}{p})^{+}}{(\frac{1}{q}%
)^{-}})<\theta <\frac{q^{-}}{r}$. H\"{o}lder's and Minkowski's inequalities
yield%
\begin{eqnarray*}
\big\||c2^{j\alpha (\cdot )}(\varphi _{j}\ast f)|^{q(\cdot )}\big\|_{\frac{%
p(\cdot )}{q(\cdot )}} &\leq &\Big(\int_{2^{-j-2}}^{2^{-j+2}}\big\||k_{\tau
}^{\ast ,a}\tau ^{-\alpha (\cdot )}f|^{q(\cdot )}\big\|_{\frac{p(\cdot )}{%
q(\cdot )}}^{\frac{1}{\theta }}\frac{d\tau }{\tau }\Big)^{\theta } \\
&\leq &\int_{2^{-j-2}}^{2^{-j+2}}\big\||k_{\tau }^{\ast ,a}\tau ^{-\alpha
(\cdot )}f|^{q(\cdot )}\big\|_{\frac{p(\cdot )}{q(\cdot )}}\frac{d\tau }{%
\tau }.
\end{eqnarray*}%
We obtain%
\begin{equation*}
\sum_{j=2}^{\infty }\big\||c2^{j\alpha (\cdot )}(\varphi _{j}\ast
f)|^{q(\cdot )}\big\|_{\frac{p(\cdot )}{q(\cdot )}}\leq 1,
\end{equation*}%
with an appropriate choice of $c>0$ such that the left hand side of %
\eqref{main-est} it at most one. Similarly we obtain%
\begin{equation*}
\big\||c\varphi _{j}\ast f|^{q(\cdot )}\big\|_{\frac{p(\cdot )}{q(\cdot )}%
}\leq 1,\quad j=0,1.
\end{equation*}%
The desired estimate follows by the scaling argument.

\noindent \textit{Step} 3.\ We will prove in this step that for all $f\in
B_{p(\cdot ),q(\cdot )}^{\alpha (\cdot )}$ the following estimates are true:%
\begin{equation}
\big\|f\big\|_{B_{p(\cdot ),q(\cdot )}^{\alpha (\cdot )}}^{\prime }\lesssim %
\big\|f\big\|_{B_{p(\cdot ),q(\cdot )}^{\alpha (\cdot )}}\lesssim \big\|f%
\big\|_{B_{p(\cdot ),q(\cdot )}^{\alpha (\cdot )}}^{\prime }.
\label{fin-est1}
\end{equation}%
Let $\{\mathcal{F}\varphi _{j}\}_{j\in \mathbb{N}_{0}}$ be a resolution of
unity. The first inequality follows by the chain of the estimates%
\begin{equation*}
\big\|f\big\|_{B_{p(\cdot ),q(\cdot )}^{\alpha (\cdot )}}^{\prime }\lesssim %
\big\|\varphi _{0}^{\ast ,a}f\big\|_{p(\cdot )}+\Big\|\big(\varphi
_{j}^{\ast ,a}2^{j\alpha (\cdot )}f\big)_{j\geq 1}\Big\|_{\ell ^{q(\cdot
)}(L^{p(\cdot )})}\lesssim \big\|f\big\|_{B_{p(\cdot ),q(\cdot )}^{\alpha
(\cdot )}},
\end{equation*}%
where the first inequality is \eqref{key-est1}, see Step 1 and the second
inequality is obvious, see \cite{D3}. Now the second inequality in %
\eqref{fin-est1} can be obtained by the following chain of the estimates 
\begin{equation*}
\big\|f\big\|_{B_{p(\cdot ),q(\cdot )}^{\alpha (\cdot )}}\lesssim \big\|%
\varphi _{0}\ast f\big\|_{p(\cdot )}+\Big\|\big(2^{j\alpha (\cdot )}(\varphi
_{j}\ast f)\big)_{j\geq 1}\Big\|_{\ell ^{q(\cdot )}(L^{p(\cdot )})}\lesssim %
\big\|f\big\|_{B_{p(\cdot ),q(\cdot )}^{\alpha (\cdot )}}^{\prime },
\end{equation*}%
where the first inequality is obvious and the second inequality is %
\eqref{fin-est}, see Step 2. Thus, Theorem \ref{loc-mean-char} is proved.
\end{proof}

\end{document}